\documentclass[15pt]{amsart}
\usepackage{amscd, amsmath, amsthm, amssymb, MnSymbol, stmaryrd}

\usepackage{comment}

\usepackage{color}

\usepackage[all]{xy}

\newcommand\WDiv{{\mathrm{WDiv}}}
\newcommand\Cl{{\mathrm{Cl}}}
\newcommand\Pic{{\mathrm{Pic}}}
\newcommand\Eff{{\mathrm{Eff}}}

\newcommand\Nef{{\mathrm{Nef}}}
\newcommand\Mov{{\mathrm{Mov}}}

\newcommand\SAmp{{\mathrm{SAmp}}}

\newcommand{\Spec}{\text{Spec}}

\newtheorem{theorem}{Theorem}[section]
\newtheorem{lemma}[theorem]{Lemma}
\newtheorem{proposition}[theorem]{Proposition}

\newtheorem{question}[theorem]{Question}
\newtheorem{corollary}[theorem]{Corollary}
\theoremstyle{definition}     
\newtheorem{definition}[theorem]{Definition}

\theoremstyle{remark}
\newtheorem{remark}[theorem]{Remark}
\numberwithin{equation}{section}

\begin{document}

\date\today

\title[Combinatorially minimal Mori dream surfaces]{Combinatorially minimal \\ Mori dream surfaces of general type}

\author[J. Keum]{JongHae Keum}
\address{School of Mathematics, Korea Institute for Advanced Study, Seoul 02455, Republic of Korea} \email{jhkeum@kias.re.kr}

\author[K.-S. Lee]{Kyoung-Seog Lee}
\address{Institute of the Mathematical Sciences of the Americas, University of Miami, Department of Mathematics, 1365 Memorial Drive, Ungar 503, Coral Gables, FL 33146} \email{kyoungseog02@gmail.com}

\thanks{} 



\begin{abstract} 
In this paper, we suggest a new approach to study minimal surfaces of general type with $p_g=0$ via their Cox rings, especially using the notion of combinatorially minimal Mori dream space introduced by Hausen in \cite{Hausen}. First, we study general properties of combinatorially minimal Mori dream surfaces. Then we discuss how to apply these ideas to the study of minimal surfaces of general type with $p_g=0$ which are very important but still mysterious objects. In our previous paper \cite{Keum-Lee1}, we provided several examples of Mori dream surfaces of general type with $p_g=0$ and computed their effective cones explicitly. In this paper, we study their fibrations, explicit combinatorially minimal models and discuss singularities of the combinatorially minimal models. We also show that many minimal surfaces of general type with $p_g=0$ arise from the minimal resolutions of combinatorially minimal Mori dream surfaces.
\end{abstract}
\maketitle


\section{Introduction}

Enriques-Kodaira classification tells us a rough geography of compact complex surfaces. It seems that we have a relatively good understanding of surfaces with Kodaira dimension less than or equal to 1. However, surfaces of general type are still very mysterious objects although they have been studied for a long time. Even for surfaces of general type with fixed numerical invariants, we do not know how to understand them in a systematic way except for a few cases. Especially, it is desirable to classify minimal surfaces of general type with $p_g=0$ since they can be considered as surfaces of general type having minimal numerical invariants. Since the first examples of these surfaces were found by Godeaux and Campedelli in 1930s, they have been important objects in algebraic geometry. In 1980, Mumford asked the following question (cf. \cite{BCP}).

\begin{question}[Mumford]
Can a computer classify all surfaces of general type with $p_g=0$?
\end{question}

Recently, there have been lots of new developments in the theory of minimal surfaces of general type with $p_g=0$. Many new examples were found, some of them with special properties were classified and several new phenomena were observed (cf. \cite{BCP}). However it seems that we are still far from the complete classification of these surfaces. \\

On the other hand, there have been lots of new progress in birational geometry and the theory of Cox rings in these days. It is desirable to use these recent progress to have a deeper understanding of surfaces of general type. Therefore it is a natural attempt to try to understand surfaces of general type via their birational invariants and Cox rings. It seems that an ideal class of objects to study using these tools are Mori dream surfaces. Especially, minimal surfaces of general type with $p_g=0$ are natural objects to study since they have relatively small Picard numbers and there are some similarities between these surfaces with del Pezzo surfaces. From this perspective the following questions seems to be natural.

\begin{question}
Let $X$ be a minimal surface of general type with $p_g=0.$ When the effective cone of $X$ is a rational polyhedral cone? When the nef cone of $X$ is the same as the semiample cone of $X$? When the Cox ring of $X$ is finitely generated?
\end{question}

In our previous paper \cite{Keum-Lee1}, we provided several examples of Mori dream surfaces which are minimal surfaces of general type with $p_g=0$ and $2 \leq K^2 \leq 9.$ Indeed, it turns out that many classical or well-known surfaces of general type are Mori dream surfaces. We also computed effective cones of these surfaces explicitly. It seems that there are many more Mori dream surfaces of minimal surfaces of general type with $p_g=0.$ In fact, we do not know a single example of minimal surface of general type with $p_g=0$ which is not a Mori dream space. We think that an answer of the following question can be an important step toward the classification of minimal surfaces of general type with $p_g=0.$

\begin{question}
Can we characterize or classify minimal surfaces of general type with $p_g=0$ which are Mori dream spaces?
\end{question}

Indeed, there are many attempts to classify Mori dream rational surfaces. See \cite{ADHL} and references therein for more details. Especially, it is well-known that weak del Pezzo surfaces are Mori dream spaces. Recall that every weak del Pezzo surface can be obtained from $\mathbb{P}^2$ or $\mathbb{P}^1 \times \mathbb{P}^1$ by blowing-ups. Note that $\mathbb{P}^2$ and $\mathbb{P}^1 \times \mathbb{P}^1$ are Mori dream surfaces which do not contain a contractible curve. Hausen defined the notion of combinatorially minimal Mori dream space as follows.

\begin{definition}\cite[Definition 6.1]{Hausen} 
Let $X$ be a Mori dream space. \\
(1) A divisor $D \in \Cl(X)$ is combinatorially contractible if it belongs to an extremal ray of $\Eff(X)$ and $h^0(X,nD) \leq 1$ for all $n \in \mathbb{N}.$ \\
(2) $X$ is combinatorially minimal if there is no combinatorially contractible divisor. 
\end{definition}

In \cite{Hausen}, Hausen developed the theory of neat ambient modifications of algebraic varieties with finitely generated Cox rings. Especially, he proved that any $\mathbb{Q}$-factorial projective Mori dream space can be obtained from a combinatorially minimal Mori dream space by sequences of neat ambient modifications and small $\mathbb{Q}$-factorial modifications (cf. Theorem  \ref{Hausen;mainthm}). However, it is very hard to compute the Cox ring of an algebraic variety in general and therefore it seems to be difficult to apply this rather complicated procedure for certain examples like varieties of general type explicitly. And it is not clear how to understand combinatorially minimal Mori dream spaces systematically. On the other hand, his construction works for any $\mathbb{Q}$-factorial Mori dream space, so it will be nice to use his idea to study complicated examples like surfaces of general type. In this paper, we study properties and characterization of combinatorially minimal Mori dream surfaces and how one can reach a combinatorially minimal surface from a Mori dream surface. We obtain the following result.

\begin{theorem}\label{contraction}
Let $X$ be a smooth Mori dream surface. Then we have the following. \\
(1) We can contract negative curves from $X$ to reach a combinatorial minimal Mori dream surface $X_0$ as follows. 
$$ X=X_n \to X_{n-1} \to \cdots \to X_0. $$
(2) The surface $X_0$ is normal $\mathbb{Q}$-factorial and the Picard number of $X_0$ is one or two. \\
(3) If the Picard number of $X_0$ is two, then $X_0$ has a finite morphism to $\mathbb{P}^1 \times \mathbb{P}^1.$ \\
(4) If $X$ is minimal, then $X$ is the minimal resolution of $X_0.$ \\
(5) Conversely, if $X_0$ be a normal $\mathbb{Q}$-factorial surface whose Picard number is 1 or Picard number 2 with a finite morphism to $\mathbb{P}^1 \times \mathbb{P}^1$, then $X_0$ is a combinatorially minimal Mori dream surface.
\end{theorem}

We have two different proofs of the above result. One is using Hausen's work of neat ambient modifications developed in \cite{Hausen} and the other one is using surface theory and several general results about Mori dream spaces. In this paper, we provide both proofs since we think the first proof can be easily extended to higher dimensional varieties and the second proof provide an explicit geometric way to construct combinatorially minimal surface $X_0$ from $X.$ Indeed, the above result is also true even if $X$ is a singular Mori dream surface. See Theorem \ref{contraction:proof1} for more precise statement. In this paper, we will focus on smooth examples. It is convenient to introduce the following definition for further discussions.

\begin{definition}
Let $X_0$ be a combinatorially minimal Mori dream surface and $X$ be a resolution of singularities of $X_0.$ Then we call $X_0$ a combinatorially minimal model of $X.$
\end{definition}

Then we discuss how to use the above ideas to study minimal surfaces of general type with $p_g=0.$ First, we need some examples which show how the above procedures look like. In \cite{Keum-Lee1}, we provided many examples of smooth projective Mori dream surfaces of general type with $p_g=0$ as follows.

\begin{theorem}\cite{Keum-Lee1}\label{Keum-Lee1main}
The following minimal surfaces of general type with $p_g=0$ are Mori dream surfaces and their negative curves are listed as follows. Here $m(C^2, p_a(C))$ means $m$ copies of $(C^2, p_a(C)).$
\begin{enumerate}
\item Fake projective planes $($all have $K^2=9),$ none
\item Surfaces isogenous to a higher product of unmixed type $($all have $K^2=8),$ none
\item Inoue surfaces with $K^2=7,$ $2(-1,1), (-1,2)$
\item Chen's surfaces with $K^2=7,$ $(-1,1), (-1,2), (-1,3), (-4,2)$ 
\item Kulikov surfaces with $K^2=6,$ $6(-1,1)$
\item Burniat surfaces with $2 \leq K^2 \leq 6$
\begin{enumerate}
\item $6(-1,1)$ if $K^2 = 6;$
\item $9(-1,1), (-4,0)$ if $K^2 = 5;$
\item $12(-1,1), 4(-4,0)$ if non-nodal with $K^2 = 4;$
\item $10(-1,1), 2(-4,0), (-2,0)$ if nodal with $K^2 = 4;$
\item $9(-1,1), 3(-4,0), 3(-2,0)$ if $K^2 = 3;$
\item $6(-1,1), 6(-2,0), 4(-4,0)$  if $K^2 = 2;$
\end{enumerate}
\item Product-quotient surfaces with $K^2=6, G=D_4 \times \mathbb{Z}/2\mathbb{Z},$ $2(-2, 0), (-1,1), (-1, 2)$
\item A family of Keum-Naie surfaces which are product-quotient surfaces with $K^2=4, G=\mathbb{Z}/4\mathbb{Z} \times \mathbb{Z}/2\mathbb{Z},$ $4(-1, 1), 4(-2, 0)$
\end{enumerate}
\end{theorem}

From the perspective of Theorem \ref{contraction}, we perform combinatorial contractions for the above examples. We study fibrations and contractions of the above examples and obtain the following results.

\begin{theorem}\label{CMM:example}
Let $X$ be a Mori dream surfaces of general type with $p_g=0$ in the above list. Let $X=X_n \to X_{n-1} \to \cdots \to X_0$ be a combinatorial contraction to a combinatorially minimal Mori dream surface $X_0.$ In each case (1)-(8) we have the following. 
\begin{enumerate}
\item $X$ itself is a combinatorially minimal Mori dream surface.
\item $X$ itself is a combinatorially minimal Mori dream surface.
\item We have $\rho(X_0) = 2$ and there are three ways to obtain combinatorially minimal models.
\item We have $\rho(X_0) = 2$ and there are four ways to obtain combinatorially minimal models.
\item We have $\rho(X_0)=1$ or $2.$ 
\item We have $\rho(X_0)=1$ or $2.$
\item We have $\rho(X_0) = 2$ and there are six ways to obtain combinatorially minimal models. 
\item We have $\rho(X_0) = 1$ or $2.$
\end{enumerate}
Moreover, we can explicitly construct each $X_0$ from the above list.
\end{theorem}

Recently, Frapporti and the second named author studied effective, nef and semiample cones of surfaces isogenous to a product of mixed type with $p_g=0$ and provided more examples of Mori dream surfaces of general type in \cite{FL}. From the above perspective, the main result of \cite{FL} can be summarized as follows.

\begin{theorem}\cite{FL}
All reducible fake quadrics are Mori dream surfaces. Moreover, they are all combinatorially minimal.
\end{theorem}

In general, it is a difficult task to determine whether a given surface is Mori dream or not. There are many surfaces which we do not know whether they are Mori dream surfaces or not yet. However, we can still ask whether they are resolutions of certain combinatorially minimal Mori dream surfaces.

\begin{question}
Let $X$ be a minimal surfaces of general type with $p_g=0.$ Can we contract negative curves on $X$ to reach a combinatorially minimal Mori dream surface? In other words, can we realize $X$ as the minimal resolution of a combinatorially minimal Mori dream surface? 
\end{question}

We can provide several examples of minimal surfaces of general type with $p_g=0$ which are minimal resolutions of combinatorially minimal Mori dream surfaces.

\begin{theorem}\label{CMM:example2}
The following smooth minimal surfaces of general type with $p_g=0$ are minimal resolutions of combinatorially minimal Mori dream surfaces of general type.
\begin{enumerate}
\item Minimal product-quotient surfaces classified in \cite{BCGP, BP}.
\item Minimal resolutions of $\mathbb{Q}$-homology planes.
\item Surfaces with $K^2=7$ constructed by Y. Chen and Y. Shin (cf. \cite{CS}).
\item Numerical Campedelli surfaces constructed by Y. Chen and Y. Shin (cf. \cite{CS}).
\item Surfaces with $K^2=6$ constructed by M. Inoue (cf. \cite{Inoue}) and later described by M. Mendes Lopes and R. Pardini (cf. \cite{MLP04}).
\end{enumerate}
\end{theorem}

From the above discussions, we can see that there are many examples of combinatorially minimal models of smooth minimal surface of general type with $p_g=0.$ Then we have the following natural question.

\begin{question}
Let $X_0$ be a combinatorially minimal model of a smooth minimal surface of general type with $p_g=0.$ What kinds of singularities $X_0$ can have?
\end{question}

When $X_0$ has only nodal singularities, then there are lots of study of such $X_0.$ For example, if $X_0$ has at worst nodal singularities and $\rho(X_0)=1,$ then we see that $X_0$ is a fake projective plane by \cite[Theorem 1.1]{Keum10}. From the product-quotient surface examples, we see that $X_0$ can have singular points other than nodal singularities. We can also see that $X_0$ can have singular points worse than quotient singularities from our previous work \cite{Keum-Lee1}. It will be very interesting if we can classify the classes of singularities of combinatorially minimal models of smooth minimal surfaces of general type with $p_g=0$. We can provide possible list of singularities when $X_0$ have at worst quotient singularities using similar arguments used in \cite{HK, HKO}. Indeed, list of possible singularities of $\mathbb{Q}$-homology planes can have was studied in \cite{HK, HKO}. We can partially extend some results in \cite{HK, HKO} and obtain the list of possible singularities of $X_0$ can have as follows.

\begin{theorem}
Let $X_0$ be a combinatorially minimal Mori dream surface having at worst quotient singularities whose minimal resolution is a minimal surfaces of general type with $p_g=0.$ Then there are finitely many possible types of singularities $X_0$ can have. Moreover we have the following cases. \\

(1) If the Picard number of $X_0$ is one, then the number of singular points is less than or equal to four. If the number of singular points of $X_0$ is four then the singularities of $X_0$ is $4A_2$ or $2A_1 \oplus 2A_3.$

\bigskip

(2) If the Picard number of $X_0$ is two, then the number of singular points is less than or equal to six. If the number of singular points of $X_0$ is six, then the possible type of singularities of $X_0$ is one of $6A_1, 4A_1+A_2+\frac{1}{3}(1,1).$
If the number of singular points of $X_0$ is five, then the possible type of singularities of $X_0$ is one of $3A_1+\frac{1}{4}(1,1)+\frac{1}{8}(1,3), 4A_1+A_3.$
\end{theorem}

\begin{remark}
Among the above list, $4A_2$, $2A_1 \oplus 2A_3$, $6A_1$, $4A_1+A_2+\frac{1}{3}(1,1)$ are supported by examples. We do not know whether the type $3A_1+\frac{1}{4}(1,1)+\frac{1}{8}(1,3), 4A_1+A_3$ in the above list are supported by examples or not. 
\end{remark}

We can also obtain a list of the possible type of singularities $X_0$ can have when the number of singular points is small. However there are many more cases we need to consider and it seems that one need to find stronger restrictions to remove redundant cases. We left this problem for future research. \\

By analyzing combinatorially minimal models of several minimal surfaces of general type with $p_g=0,$ we know that they often have singularities worse than quotient singularities. Therefore we have the following natural question.

\begin{question}
Let $X$ be a smooth minimal Mori dream surface of general type with $p_g=0.$ Is there at least one $X_0$ having at worst Du Bois singularities?
\end{question}

By analyzing the examples in our previous work, we obtain the following result.

\begin{proposition}
Let $X$ be a Mori dream surface of general type with $p_g=0$ in the list of Theorem \ref{CMM:example} and Theorem \ref{CMM:example2}. Then there is at least one $X_0$ having at worst log canonical hence Du Bois singularities.
\end{proposition}

We think studying the above question will be helpful to answer the following questions.

\begin{question}
Can we classify Mori dream surfaces which are combinatorially minimal? Moreover can we classify smooth minimal Mori dream surfaces with $p_g=0$ via studying their combinatorially minimal models? Does every smooth minimal surfaces of general type with $p_g=0$ can be obtained as the minimal resolution of a combinatorially minimal Mori dream surface?
\end{question}

Let us give brief outline of this paper. In \S\ref{s2}, we recall some preliminaries about Mori dream spaces and Cox rings. And we briefly explain Hausen's work studying neat ambient modifications of Mori dream spaces (cf. \cite{Hausen}) for the convenience of the reader. In \S\ref{s3}, we discuss general properties of combinatorially minimal models of Mori dream surfaces. In \S\ref{s4}, we study fibrations of Mori dream surfaces and discuss how to obtain combinatorially minimal models from smooth Mori dream surfaces. Then we analyze combinatorially minimal models of Mori dream surfaces of general type with $p_g=0$ we studied in our previous paper \cite{FL, Keum-Lee1}. In \S\ref{s5}, we present several examples of minimal surfaces of general type with $p_g=0$ which are minimal resolutions of combinatorially minimal Mori dream surfaces. In \S\ref{s6} and \S\ref{s7}, we study possible singularity types of combinatorially minimal Mori dream surfaces having at worst quotient singularities. In \S\ref{s8}, we discuss the singularities of combinatorially minimal models of minimal surfaces of general type with $p_g=0.$ In \S\ref{s9}, we raise several natural questions and discuss possible further directions.

\bigskip

{\bf Notations.} We will work over $\mathbb{C}.$ When $G$ is a $\mathbb{Z}$-module, then $G_{\mathbb{R}}$ (resp. $G_{\mathbb{Q}}$) will mean $G \otimes_{\mathbb{Z}} {\mathbb{R}}$ (resp. $G \otimes_{\mathbb{Z}} {\mathbb{Q}}$). An algebraic variety will mean an integral seperated scheme of finite type over $\mathbb{C}$. A Mori dream space will mean a normal $\mathbb{Q}$-factorial projective space whose Cox ring is finitely generated. We will use the following notations for a normal projective variety $X.$ \\
When $D_1, D_2$ are two divisors on $X.$ We write $D_1 \sim D_2$ (resp. $D_1 \sim_{num} D_2$) to denote that they are linearly (resp. numerically) equivalent. \\
$K_X$ : the canonical divisor on $X.$ \\
$\WDiv(X)$ : the group of Weil divisors on $X.$ \\
$\Cl(X)$ : the divisor class group of $X.$ \\
$\Pic(X)$ : the Picard group of $X.$ \\
$\rho(X)$ : the Picard number of $X.$ \\
$\Eff(X)$ : the effective cone of $X$. \\
$\Nef(X)$ : the nef cone of $X$. \\
$\Mov(X)$ : the movable cone of $X$. \\
$\SAmp(X)$ : the semiample cone of $X$.

\bigskip

{\bf Acknowledgements.} We thank Ingrid Bauer, Frederic Campana, Fabrizio Catanese, Sung Rak Choi, Stephen Coughlan, Davide Frapporti, June Huh, DongSeon Hwang, Yongnam Lee, Han-Bom Moon, Jinhyung Park and Miles Reid for helpful conversations and discussions. Part of this work was done when the second named author was a research fellow of KIAS. He was also partially supported by IBS-R003-Y1 when he was working at IBS-CGP as Young Scientist Fellow. Last but not least, he thanks Ludmil Katzarkov and Simons Foundation for partially supporting this work via Simons Investigator Award-HMS.

\bigskip

\section{Cox rings and Mori dream spaces}\label{s2}

In this section we review basic definitions and properties about Cox rings and Mori dream spaces. We also review several constructions and terminology of Hausen. We will closely follow \cite{Hausen} and use notations and constructions therein.

\subsection{Cox rings and Mori dream spaces}

There are several ways to define Cox ring of an algebraic variety. It is relatively easier to define Cox rings of algebraic varieties whose divisor class groups do not contain torsions (cf. \cite{ADHL}). Hausen defined $\Cl$-graded Cox rings for varieties whose divisor class groups contain torsions in \cite{Hausen} as follows.

\bigskip

Let $X$ be a variety with finitely generated $\Cl(X).$ Let $\mathcal{K} \subset \WDiv(X)$ be a finitely generated subgroup mapping onto $\Cl(X).$ Let us define a sheaf of $\mathcal{K}$-graded algebra $\mathcal{S}$ as follows
$$ \mathcal{S}:= \bigoplus_{D \in \mathcal{K}} \mathcal{S}_D $$
where $\mathcal{S}_D:=\mathcal{O}_X(D)$ and the multiplication is induced from the ring structure of function field of $X$. Let $\mathcal{K}^0$ be the kernel of $\mathcal{K} \to \Cl(X)$. Then a shifting family is a family of $\mathcal{O}_X$-module isomorphisms $\rho_{D^0} : \mathcal{S} \to \mathcal{S}$ such that \\
(1) $\rho_{D^0}(\mathcal{S}_D)=\mathcal{S}_{D+D^0}$ for all $D \in \mathcal{K}, D^0 \in \mathcal{K}^0$, \\
(2) $\rho_{D_1^0+D_2^0}=\rho_{D_2^0} \circ \rho_{D_1^0}$ for all $D_1^0, D_2^0 \in \mathcal{K}^0$, \\
(3) $\rho_{D^0}(fg)=f \rho_{D^0}(g)$ for all $D^0 \in \mathcal{K}^0$ and two homogeneous $f,g$.

\bigskip

Let $\mathcal{I}$ be a sheaf generated by all sections of the form $f-\rho_{D^0}(f)$ where $f$ is homogeneous and $D^0$ is an element of $\mathcal{K}^0$. Then we can define Cox sheaf and Cox ring of $X$ as follows.

\begin{definition}
The Cox sheaf of $X$ is the quasi-coherent sheaf of $\Cl(X)$-graded $\mathcal{O}_X$-algebras $$ \mathcal{R}:=\mathcal{S}/\mathcal{I}. $$
The Cox ring of $X$ is the $\Cl(X)$-graded algebra $$ \mathcal{R}(X):= H^0(X,\mathcal{R}). $$
\end{definition}

Hausen proved that Cox sheaf and Cox ring of $X$ is well-defined up to isomorphism in \cite{Hausen}. \\

Hu and Keel introduced the notion of Mori dream space as follows.

\begin{definition}\cite{HK}
Let $X$ be a normal projective $\mathbb{Q}$-factorial variety with finitely generated $\Cl(X).$ We call $X$ a Mori dream space if \\
(1) $\Nef(X)$ is generated by finitely many semiample divisor classes, \\
(2) there are finitely many small $\mathbb{Q}$-factorial modifications $\phi_i : X \dashedrightarrow X_i, 1 \leq i \leq m$ such that $X_i$ are varieties wifh finitely generated $\Cl(X_i)$ satisfying the above condition and when $D$ is a movable divisor on $X$, there is a semiample divisor $D_i$ on $X_i$ with $D=\phi^*_iD_i$ for some $i$.
\end{definition}

Moreover they proved that a variety is Mori dream space if and only if its Cox ring is finitely generated in \cite{HK}.

\begin{theorem}\cite{HK}
Let $X$ be a variety with finitely generated $\Cl(X).$ Then $X$ is a Mori dream space if $\mathcal{R}(X)$ is a finitely generated $\mathbb{C}$-algebra.
\end{theorem}

Let $X$ be a Mori dream space. Then we can consider the relative spectrum $\widehat{X}:=\Spec_X(\mathcal{R})$ and let $\overline{X}=\Spec(\mathcal{R}(X))$. Then there is an open embedding $\widehat{X} \to \overline{X}$ and natural action of $H:=\Spec(\mathbb{C}[\Cl(X)])$ on $\widehat{X}$ and $\overline{X}.$ Moreover, the canonical morphism $p : \widehat{X} \to X$ is a good quotient. See \cite{Hausen, HK} for more details.

\begin{remark}
Let $\mathcal{K} \subset \WDiv(X)$ be a finitely generated subgroup such that $\mathcal{K}_{\mathbb{Q}} \cong \Cl(X)_{\mathbb{Q}}.$ One can define $\mathcal{R}_{\mathcal{K}}(X)$ as a multisection ring as follows. 
$$ \mathcal{R}_{\mathcal{K}}(X) = \oplus_{D \in \mathcal{K}} H^0(X, \mathcal{O}(D)) $$
In some literature, $\mathcal{R}_{\mathcal{K}}(X)$ is also called a Cox ring of $X.$ One can see that $\mathcal{R}_{\mathcal{K}}(X)$ is finitely generated if and only if $\mathcal{R}(X)$ is finitely generated. See \cite{Okawa} for more details. In this paper, Cox ring of $X$ will mean the Cox ring defined by Hausen in \cite{Hausen}.
\end{remark}

Let $X$ be a Mori dream space. Because $\mathcal{R}(X)$ is a finitely generated $\mathbb{C}$-algebra, we see that there is a surjective map from polynomial algebra to $\mathcal{R}(X).$ It turns out that this map induces a very special embedding which we will describe as follows.

\begin{definition}\cite{Hausen}
Let $Z$ be a toric variety with acting torus $T$ and $X \subset Z$ be an irreducible closed subvariety which is smooth in codimension 1. The embedding $i : X \hookrightarrow Z$ is called neat embedding if \\
(1) the intersection of the toric divisor $\overline{T \cdot z_i}$ with $X$ is an irreducible hypersurface for each $i,$ \\
(2) the pullback $i^* : \Cl(Z) \to \Cl(X)$ is an isomorphism.
\end{definition}

Indeed, there is a surjection $\mathbb{C} [U_1, \cdots, U_r] \to \mathcal{R}(X)$ and one can define a $K$-grading on $\mathbb{C} [U_1, \cdots, U_r]$ by $\deg(U_i):=\deg(f_i)$ and it induces an $H$-equivariant closed embedding 
$$ \overline{\iota} : \overline{X} \to \overline{Z}:=\mathbb{C}^r, z \mapsto (f_1(z),\cdots,f_r(z)). $$

This embedding induces the following diagram.

\begin{displaymath}
\xymatrix{
\widehat{X} \ar[rr] \ar[d]_{p_X}^{\sslash H} &  & \widehat{Z} \ar[d]_{p_Z}^{\sslash H}  \\
X \ar[rr] &  & Z
}
\end{displaymath}

\begin{proposition}\cite[Proposition 3.14]{Hausen}
In the above construction we have the following. \\
(1) $\widehat{X} = \overline{\iota}^{-1}(\widehat{Z})$ and $\hat{\iota} : \widehat{X} \to \widehat{Z}$ is an $H$-equivariant closed embedding. \\
(2) For $1 \leq i \leq r$, we use $D^i_Z$ be the invariant prime divisor corresponding to $e_i \in E$. Then $D_X^i = \iota^* (D^i_Z)$ for all $i$. \\
(3) The induced morphism $\iota : X \to Z$ is a neat closed embedding.
\end{proposition}

\subsection{Ambient modification}

Let $\pi : Z_1 \to Z_0$ be the toric morphism associated to a stellar subdivision of simplicial fans. Then we have the following commutative diagram of toric morphisms. See \cite[Proposition 5.2]{Hausen} for more details.

\begin{displaymath}
\xymatrix{ 
 & \ar[ld]_{\bar{\pi}_1} \bar{Z}_1 \ar[rd]^{// \mathbb{C}^*} &  \\
\bar{Z}_1 & \ar[ld]_{\hat{\pi}_1} \hat{Z}_1 \ar[rd] \ar[u] & \bar{Z}_0  \\ 
\hat{Z}_1 \ar[u] \ar[d]_{/H_1} &  & \hat{Z}_0 \ar[u] \ar[d]^{/H_0}  \\
Z_1 \ar[rr]_{\pi} &  & Z_0 
}
\end{displaymath}

Hausen defined the notion of neat ambient modification as follows.

\begin{definition}[Neat ambient modification]
Let $Z_0, Z_1$ be simplicial toric varieties with action torus $T_0, T_1$ and let $\pi : Z_1 \to Z_0$ be the toric morphism induced from a stellar subdivision of simplicial fans.
A toric morphism $\pi : Z_1 \to Z_0$ is called a neat ambient modification for $X_1 \subset Z_1$ and $X_0 \subset Z_0$ if
\begin{enumerate}
\item $X_1 \cap E$ is an irreducible hypersurface in $X_1$ intersecting $T_1 \cdot z,$
\item local equation for $E \subset Z_1$ restricts to a local equation for $D \subset X_1,$
\item $X_0 \cap \pi(E)$ is of codimension at least two in $X_0.$
\end{enumerate}
\end{definition}

The neat embedding and the above commutative diagram of toric morphisms induces a neat ambient modification. Let $\pi : Z_1 \to Z_0$ be a neat ambient modification for two irreducible subvarieties $X_0 \subset Z_0$ and $X_1 \subset Z_1.$ Then we have the following diagram. Here $\hat{Y}_1$ is the inverse image of $\hat{X}_1$ under $\bar{\pi}_1$ and $\bar{Y}_1$ is the closure of $\hat{Y}_1$ in $\bar{Z}_1.$ See \cite{Hausen} for more details.

\begin{displaymath}
\xymatrix{ 
 & \ar[ld] \bar{Y}_1 \ar[rd]^{// \mathbb{C}^*} &  \\
\bar{X}_1 & \ar[ld] \hat{Y}_1 \ar[rd] \ar[u] & \bar{X}_0  \\ 
\hat{X}_1 \ar[u] \ar[d]^{/H_1} &  & \hat{X}_0 \ar[u] \ar[d]^{/H_0}  \\
X_1 \ar[rr] &  & X_0 
}
\end{displaymath}

Now let us recall the definition of combinatorially minimal Mori dream space which was introduced by Hausen in \cite{Hausen}.

\begin{definition}\cite[Definition 6.1]{Hausen} 
Let $X$ be a Mori dream space. \\
(1) A divisor class $[D] \in \Cl(X)$ is combinatorially contractible if it belongs to an extremal ray of $\Eff(X)$ and $h^0(X,nD) \leq 1$ for all $n \in \mathbb{N}.$ \\
(2) $X$ is combinatorially minimal if there is no combinatorially contractible divisor. 
\end{definition}

We have the following simple criterions.

\begin{remark}\cite[Corollary 6.8, 6.9]{Hausen}
(1) A $\mathbb{Q}$-factorial projective Mori dream space $X$ is combinatorially minimal if and only if $\Eff(X)=\Mov(X).$ \\
(2) A $\mathbb{Q}$-factorial projective Mori dream surface $X$ is combinatorially minimal if and only if $\Eff(X)=\SAmp(X).$
\end{remark}

Hausen proved that every Mori dream space can be obtained from a combinatorially minimal Mori dream space via neat controlled ambient modification.

\begin{theorem}\cite[Theorem 6.2]{Hausen}\label{Hausen;mainthm}
Let $X$ be a projective $\mathbb{Q}$-factorial Mori dream space. Then $X$ can be obtained from a combinatorially minimal Mori dream space $X_0$ via finite sequence
$$ X=X_n' \dashedrightarrow X_n \to X_{n-1}' \dashedrightarrow X_{n-1} \to \cdots \to X_0' = X_0 $$
where $X_i' \dashedrightarrow X_i$ is a small birational transformation and $X_i \to X_{i-1}'$ is induced from a neat controlled ambient modification of $\mathbb{Q}$-factorial projective toric varieties.
\end{theorem}

See \cite{Hausen} for more details about the construction.

\bigskip

\section{Combinatorially minimal Mori dream surfaces}\label{s3}

In this section we study basic properties of combinatorially minimal Mori dream surfaces. Let us recall some basic definitions and results. \\ 

Let $X$ be a surface with finitely generated $\Cl(X)$. When is $X$ a Mori dream surface? There is a simple criterion for normal complete surfaces. Artebani, Hausen and Laface proved the following theorem in \cite{AHL}.

\begin{theorem}\cite[Theorem 2.5]{AHL}
Let $X$ be a normal complete surface with finitely generated $\Cl(X).$ Then the followings are equivalent. \\
(1) $X$ is a Mori dream space. \\ 
(2) The effective cone $\Eff(X) \subset Cl(X)_{\mathbb{Q}}$ and moving cone $\Mov(X) \subset \Cl(X)_{\mathbb{Q}}$ are polyhedral cones and $\Mov(X)=\SAmp(X).$ \\
And if one of the above two conditions holds, then $X$ is projective and $\mathbb{Q}$-factorial.
\end{theorem}

As a corollary we have the following helpful criterion of finitely generation of Cox rings of $\mathbb{Q}$-factorial surfaces.

\begin{proposition}\cite[Corollary 2.6]{AHL}
Let $X$ be a $\mathbb{Q}$-factorial projective surface with finitely generated $\Cl(X).$ Then the followings are equivalent. \\
(1) $X$ is a Mori dream space. \\ 
(2) The effective cone $\Eff(X) \subset \Cl(X)_{\mathbb{Q}}$ is a polyhedral cone and $\Nef(X)=\SAmp(X).$
\end{proposition}

Now, let us give several examples of combinatorially minimal Mori dream surfaces.

\subsection{Examples of combinatorially minimal Mori dream surfaces with $\rho=1$}

Projective plane, log del Pezzo surfaces with Picard number 1, fake projective planes, more generally $\mathbb{Q}$-homology projective planes are examples of combinatorially minimal Mori dream surfaces with $\rho=1.$ Conversely, every $\mathbb{Q}$-factorial projective variety with $\rho=1$ is a combinatorially minimal Mori dream space.

\subsection{Examples of combinatorially minimal Mori dream surfaces with $\rho=2$}

Quadric surfaces, Enriques surfaces with 8 nodes and canonical models of some product-quotient surfaces with $p_g=q=0$ are examples of combinatorially minimal Mori dream surfaces with $\rho=2.$

\subsection{Combinatorially minimal Mori dream surfaces from combinatorial contractions}

First let us recall intersections of divisors on normal surfaces. Let $X$ be a normal surface and $\pi : \widetilde{X} \to X$ be its resolution of singularities. For a $\mathbb{Q}$-divisor $D$, one can define
$$ \pi^*D = \overline{D}+\sum_{i}e_iE_i $$
where $\overline{D}$ is the strict transform of $D$ by $\pi$ and $e_i$ are the rational numbers uniquely determined by the equations of the following form
$$ \pi^*D \cdot E_j = \overline{D} \cdot E_j + \sum_{i}e_iE_i \cdot E_j=0 $$
for all $j$. For two $\mathbb{Q}$-divisors on $X$ we define $D \cdot D'$ to be
$$ D \cdot D' = \pi^*D \cdot \pi^* D'. $$

Now let us recall a contraction theorem of Sakai.

\begin{theorem}\cite{Sakai}
Let $C_1,\cdots,C_k$ be a sequence of irreducible curves on a normal surface $X.$ Then $C_1+\cdots+C_k$ can be contracted to normal points if and only if the intersection matrix $(C_i,C_j)$ is negative definite.
\end{theorem}

\begin{definition}
Let $X$ be a normal surface. A negative curve is an irreducible reduced curve on $X$ with negative self-intersection.
\end{definition}

From the above theorem we can have the following criterion.

\begin{lemma}
Let $X$ be a normal $\mathbb{Q}$-factorial Mori dream surface. Then a negative curve $D$ on $X$ induces a combinatorially contractible divisor class $[D]$. Conversely, if an irreducible reduced curve $D$ on $X$ gives a combinatorially contractible divisor class $[D]$, then $D$ is a negative curve.
\end{lemma}
\begin{proof}
If an irreducible reduced curve $D$ on $X$ gives a combinatorially contractible divisor class $[D]$, then we can contract $D$ to reach to a normal $\mathbb{Q}$-factorial Mori dream surface. Therefore we can see that $D$ is a negative curve from Sakai's theorem. Now let $D$ be a negative curve on $X$. We need to show that $h^0(X, \mathcal{O}(nD)) \leq 1$ for all $n$. Let $\pi : \widetilde{X} \to X$ be a resolution of singularities and $\pi^*D$ be the pullback of $D$. We may assume that $D$ is an effective Cartier divisor by replacing it by its suitable multiple. Then the Zariski decomposition of $n \pi^*D$ is $0 + n \pi^*D$ since $n \pi^*D$ is negative definite. From \cite[Proposition 2.3.21]{Lazarsfeld}, we see that $h^0(\widetilde{X}, \mathcal{O}(n \pi^*D)) = 1$ for all $n$ and hence we have $h^0(X, \mathcal{O}(nD)) = 1$ for all $n$ from projection formula. From this we can see that $[D]$ lies on an extremal ray of $\Eff(X)$. Therefore $D$ induces a combinatorially contractible divisor class $[D]$. 
\end{proof}

Let us recall semiample fibrations theorem. See \cite{Lazarsfeld} for more details.

\begin{definition}
An algebraic fiber space is a projective surjective morphism $f: X \to Y$ between algebraic varieties such that $f_* \mathcal{O}_X = \mathcal{O}_Y$.
\end{definition}

\begin{lemma}\cite[Example 2.1.15]{Lazarsfeld}
Let $f : X \to Y$ be an algebraic fiber space. If $X$ is normal, then $Y$ is also normal.
\end{lemma}

\begin{theorem}\cite[Theorem 2.1.27]{Lazarsfeld}[Semiample fibrations]\label{semiamplefibrations} Let $X$ be a variety and $L$ be a semiample line bundle on $X$. Then there is an algbraic fiber space
$$ \phi : X \to Y $$
such that for sufficiently large $m$,
$$ Y_m=Y, ~~~ \mathrm{and} ~~~ \phi_m=\phi. $$
Furthermore, there is an ample line bundle $A$ on $Y$ such that $\phi^*A=L^{\otimes f}$ where $f$ is the exponent of $M(X,L).$
\end{theorem}

Then we have the following.

\begin{lemma}
Let $X_0$ be a combinatorial minimal Mori dream surface. Then the Picard number of $X_0$ is less than or equal to $2.$ Moreover if the Picard number of $X_0$ is 2, then there is a finite morphism to $\mathbb{P}^1 \times \mathbb{P}^1.$
\end{lemma}
\begin{proof}
Suppose that $\rho(X_0) \geq 3.$ Then $\Nef(X_0)$ has at least three extremal rays. Because $X_0$ is a Mori dream space, each of these nef divisors are semiample and induce a map $X_0 \to Y.$ We may assume that $Y$ is a normal variety via Stein factorizaion. When $Y$ is a surface there is a negative curve contracted by $X_0 \to Y$ and it is impossible since $X_0$ is combinatorially minimal. Suppose that all such $Y$ are curves. Then $Y$ is isomorphic to $\mathbb{P}^1$ since $q(X_0)=0.$ Let us consider two facets of $\Eff(X_0)$ whose intersection is a codimension two cone. By taking dual, we have two extremal nef divisors. These semiample divisors induce two fibrations to $\mathbb{P}^1.$ Therefore we have a map $X_0 \to \mathbb{P}^1 \times \mathbb{P}^1.$ Then there is a $E$ contained in the codimension 2 cone and it is contracted by this map. Because $X_0$ is combinatorially minimal, $E$ cannot be a negative curve. Therefore we see that when $X_0$ is combinatorially minimal Mori dream surface, then $\rho(X_0) \leq 2$. \\

Now suppose that $\rho(X_0)=2$. From the above discussion, we see that there is a quasi-finite morphism $X_0 \to \mathbb{P}^1 \times \mathbb{P}^1$. Since both $X_0$ and $\mathbb{P}^1 \times \mathbb{P}^1$ are projective, we see that it is finite.
\end{proof}

Now we obtain the following result.

\begin{theorem}\label{contraction:proof1}
Let $X$ be a Mori dream surface. Then we have the following. \\
(1) We can contract negative curves from $X$ to obtain a combinatorial minimal Mori dream surface $X_0$ as follows. 
$$ X=X_n \to X_{n-1} \to \cdots \to X_0. $$
(2) The surface $X_0$ is normal projective $\mathbb{Q}$-factorial and the Picard number of $X_0$ is one or two. \\
(3) If the Picard number of $X_0$ is two, then $X_0$ has a finite morphism to $\mathbb{P}^1 \times \mathbb{P}^1.$ \\
(4) If $X$ is a smooth minimal surface then $X$ is the minimal resolution of $X_0.$ \\
(5) Conversely, if $X_0$ is a normal projective $\mathbb{Q}$-factorial surface with $\rho=1$ or $\rho=2$ having a finite morphism to $\mathbb{P}^1 \times \mathbb{P}^1,$ $X_0$ is a combinatorially minimal Mori dream surface.
\end{theorem}
\begin{proof}
Let us prove the assertion (5) first. Every normal projective $\mathbb{Q}$-factorial surface $X_0$ with $\rho=1$ is a combinatorially minimal Mori dream surface since $\Eff(X_0), \Nef(X_0)$ and $\SAmp(X_0)$ are all the same 1-dimensional ray generated by an ample divisor. Let $X_0$ be a normal projective $\mathbb{Q}$-factorial surface with $\rho=2$ having a finite morphism to $\mathbb{P}^1 \times \mathbb{P}^1.$ Let $D_1$ (resp. $D_2$) be the Cartier divisor of $X_0$ corresponds to the pullback of $\mathcal{O}(1,0)$ (resp. $\mathcal{O}(0,1)$). Then $D_1, D_2$ are semiample divisors and we have $D^2_1=D^2_2=0, D_1 \cdot D_2 > 0.$ Let $D=a_1D_1+a_2D_2$ be a nef divisor. Then we see that $a_1, a_2 \geq 0$ by intersecting $D$ with two effective divisors $D_1$ and $D_2.$ Therefore we have the following inclusions
$$ \mathbb{R}_{\geq 0} \langle D_1, D_2 \rangle \subset \SAmp(X_0) \subset \Nef(X_0) \subset \mathbb{R}_{\geq 0} \langle D_1, D_2 \rangle $$
Therefore we have $\Nef(X_0)=\SAmp(X_0)=\mathbb{R}_{\geq 0} \langle D_1, D_2 \rangle.$ Let $D=a_1D_1+a_2D_2$ be an effective divisor. By taking intersection with two nef divisors $D_1, D_2$ we have $a_1, a_2 \geq 0$. Therefore we have $\Eff(X_0) = \mathbb{R}_{\geq 0} \langle D_1, D_2 \rangle$ from the following inclusions. 
$$ \mathbb{R}_{\geq 0} \langle D_1, D_2 \rangle  \subset \Eff(X_0) \subset \mathbb{R}_{\geq 0} \langle D_1, D_2 \rangle $$
Therefore we see that $X_0$ is a combinatorially minimal Mori dream surface. \\

Now let us prove the assertions (1), (2) and (3). From Hausen's theorem we see that $X$ can be obtained from a combinatorially minimal Mori dream surface $X_0$ via finite sequence
$$ X=X_n' \dashedrightarrow X_n \to X_{n-1}' \dashedrightarrow X_{n-1} \to \cdots \to X_0' = X_0 $$
where $X_i' \dashedrightarrow X_i$ is a small $\mathbb{Q}$-factorial modification and $X_i \to X_{i-1}'$ is induced from a neat controlled ambient modification of $\mathbb{Q}$-factorial projective toric varieties. Because $X$ is a surface there is no small $\mathbb{Q}$-factorial modification and we have the following sequence of morphisms 
$$ X=X_n \to X_{n-1} \to \cdots \to X_0. $$

From the construction, we see that each $X_i$ is projective (since it is neatly embedded in a projective toric variety) and we may assume that each $X_i$ is normal by taking normalization for each step. Then we see that each $X_i$ is normal projective $\mathbb{Q}$-factorial from \cite[Theorem 9.3]{Okawa} and \cite[Theorem 2.5]{AHL}. From Sakai's theorem and construction, we see that each $X_i \to X_{i-1}$ is contraction of a negative curve on $X_i.$ Then we see that $\rho(X_i) = \rho(X_{i-1})+1$ for each $i$ from \cite[Proposition II.6.5]{Hartshorne} since we are contracting a negative curve. From the previous Lemma, we can always obtain $X_j$ such that $\rho(X_j)=2.$ If $X_j$ is combinatorially minimal, then we have $X_j = X_0$ and there is a finite morphism $X_0 \to \mathbb{P}^1 \times \mathbb{P}^1$ from the above Lemma. If $X_j$ is not combinatorially minimal, we can contract a negative curve on $X_j$ to obtain $X_j \to X_0$ where $\rho(X_0)=1.$ From (5), we see that $X_0$ is combinatorially minimal. Therefore we prove the assertions (1), (2) and (3). \\

Let $X'$ be the minimal resolution of $X_0.$ By the definition of minimal resolution there is a morphism $X \to X'$ and this morphism is an isomorphism if $X$ is minimal surface. Therefore we have the assertion (4).
\end{proof}

\begin{remark}
(1) The above result can be considered as a generalization of the fact that we can contract $(-1)$-curves on any del Pezzo surface to reach $\mathbb{P}^2$ or $\mathbb{P}^1 \times \mathbb{P}^1.$ \\ 
(2) Note that minimal resolution of a combinatorially minimal model is not always a Mori dream space. See \cite{HP1} for an example and more details.
\end{remark}

\bigskip

\section{Examples of combinatorially minimal models of Mori dream surfaces of general type with $p_g=0$}\label{s4}

In our previous paper \cite{Keum-Lee1}, we provided several examples of Mori dream surfaces of general type with $p_g=0.$ In this section, we will discuss fibrations and combinatorially minimal models of these surfaces. We will use the same notations used in \cite{Keum-Lee1}. \\

Let $X$ be a Mori dream surface. Our Theorem \ref{contraction:proof1} guarantees that we can contract negative curves on $X$ to reach a combinatorially minimal model $X_0$ of $X$. Here, we give an alternative proof of the Theorem \ref{contraction:proof1} which also gives a practical way to construct combinatorially minimal Mori dream surfaces explicitly. \\

We recall some basic definitions and properties about convex polyhedral cones. We will follow \cite{Fulton93}. Let $N_1$ be a finite dimensional vector space and let $N^1$ be the dual space with a nondegenerate pairing $\langle - , - \rangle : N_1 \times N^1 \to \mathbb{R}$.

\begin{definition}
Let $\sigma$ be a rational polyhedral cone in a finite dimensional vector space $N_1.$ A face $\tau$ of $\sigma$ is the intersection of $\sigma$ with a supporting hyperplane $u^{\perp}$ for a $u \in N^1$, i.e., $\tau = \sigma \cap u^{\perp} = \{ v \in \sigma ~ | ~ \langle u, v \rangle = 0 \}$. A facet is a face of codimension 1.
\end{definition}

Let us recall the following useful facts about rational polyhedral cones.

\begin{remark}\cite{Fulton93}
(1) Any proper face is the intersection of all facets containing it. \\
(2) If $\tau$ is a face of $\sigma$, then $\sigma^\vee \cap \tau^{\perp}$ is a face of $\sigma^\vee$ and we have $\dim(\tau) + \dim(\sigma^\vee \cap \tau^{\perp})=\dim N_1$. From here, we have a one-to-one order-reversing correspondence between the faces of $\sigma$ and the faces of $\sigma^\vee$.
\end{remark}

Now let us apply the above properties to the study of $\Eff(X)$, $\Nef(X)$ and $\SAmp(X)$. It is easy to see the following.

\begin{lemma}
Let $F$ be a facet of $\Eff(X)$ and $D$ be a nef divisor which is orthogonal to $F$. Then $[D]$ lies on an extremal ray of $\Nef(X).$
\end{lemma}
\begin{proof}
Let $u_1, u_2 \in \Nef(X)$ such that $[D]=u_1+u_2$. For any $v \in F,$ we have $0 = \langle [D], v \rangle = \langle u_1, v \rangle + \langle u_2, v \rangle \geq 0$ and hence $\langle u_1, v \rangle = \langle u_2, v \rangle = 0$. Therefore we see that both $u_1$ and $u_2$ are multiples of $[D]$.
\end{proof}

\begin{proposition}Let $X$ be a smooth projective surface and $D$ be a semiample divisor on $X$. Let $\phi : X \to Y$ be the semiample fibration defined by (multiple of) $D$ (cf. Theorem \ref{semiamplefibrations}) where $Y$ is a normal surface. Then $\phi$ contracts all negative curves whose intersections with $D$ are zero.
\end{proposition}
\begin{proof}
Let $C$ be a reduced irreducible curve on $X$ with negative self-intersection number. There is an irreducible curve $C'$ in $|mD|$ for sufficiently large $m$ whose image lies on the smooth open subset of $Y.$ Because $C$ and $C'$ are irreducible curves which are different, we have $C \cdot C' \geq 0$ and $C \cdot C'=0$ if they are disjoint. Therefore we see that $\phi$ contracts $C.$
\end{proof}

Let us recall the definition of fibration.

\begin{definition}[fibration]
Let $X$ be an algebraic surface. A fibration of $X$ is a proper map $f : X \to C$ such that all fibers are connected and $C$ is an algebraic curve.
\end{definition}

When $X$ is a regular surface and $f : X \to C$ with normal algebraic curve $C,$ then we have $C \cong \mathbb{P}^1.$

\begin{proposition}
Let $X$ be a surface with $q=0.$ There is a one-to-one correspondence between fibrations and semiample line bundles whose self-intersections are zero.
\end{proposition}
\begin{proof}
Let $\phi : X \to C$ be a fibration. Via Stein factorization, we may assume that $C$ is normal and hence smooth. Because $q=0$ we have $C \cong \mathbb{P}^1.$ By pulling back $\mathcal{O}_{\mathbb{P}^1}(1)$ to $X,$ we obtain a semiample line bundle whose self-intersection is zero. Conversely, let $L$ be a semiample line bundle with self-intersection number zero. From Theorem \ref{semiamplefibrations}, we see that there is a morphism $\phi : X \to Y.$ Because $L^2=0,$ we see that $Y$ is a curve. Therefore we obtain a fibration.
\end{proof}

\begin{proposition}\label{Pic=1,2}
Let $X$ be a Mori dream surface with $\rho = 2.$ Then there is a morphism $X \to X_0$ where $X_0$ is a normal projective $\mathbb{Q}$-factorial variety such that $\rho(X_0)=1$ or $\rho(X)=2$ with a finite morphism to $\mathbb{P}^1 \times \mathbb{P}^1.$
\end{proposition}
\begin{proof}
Let $D_1, D_2$ be two nef divisors such that $[D_1], [D_2]$ lie on the two extremal rays in $\Nef(X)$. Let $E_1, E_2$ be two effective divisors such that $[E_1], [E_2]$ lie on the two extremal rays in $\Eff(X)$ such that $D_i \cdot E_i=0$ for $i=1,2$. Suppose that one of $D^2_i > 0$. Without loss of generality we may assume that $D^2_1 > 0.$ Then multiple of $D_1$ gives a morphism $X \to X_0$ contracting $E_1$. Suppose that $D^2_1=D^2_2=0$. Then multiples of $D_i$ induces a morphism $X \to \mathbb{P}^1$ and these two fibrations induce a morphism $X \to \mathbb{P}^1 \times \mathbb{P}^1$. Then we can show that $\Eff(X) = \Nef(X) = \SAmp(X)$ and especially, we can see that there is no negative curve on $X.$ We see that the morphism $X \to \mathbb{P}^1 \times \mathbb{P}^1$ is a quasi-finite morphism since there is no negative curve on $X.$ Because $X$ is projective, we see that the morphism $X \to \mathbb{P}^1 \times \mathbb{P}^1$ is projective.
\end{proof}

\begin{lemma}
Let $\sigma$ be a face of codimension two and let $F_1, F_2$ be two facets of $\Eff(X)$ containing $\sigma$. Let $D_i$ be a nef divisor which is orthogonal to $F_i$. Suppose that $D^2_1=D^2_2=0$. Then there is a morphism $X \to \mathbb{P}^1 \times \mathbb{P}^1$ contracting $\rho-2$ curves. 
\end{lemma}
\begin{proof}
Because $\sigma$ is of codimension two, there are at least $\rho-2$ linearly independent curves in $\Eff(X)$. The morphism $X \to \mathbb{P}^1 \times \mathbb{P}^1$ contracts negative curves in $\sigma$.
\end{proof}

\begin{proposition}\cite[Proposition 1.1]{AL}
Let $X$ be a smooth projective surface with $\rho(X) \geq 3$ such that $\Eff(X)$ is rational polyhedral. Then we have 
$$ \Eff(X) = \sum_{E ~ : ~ \mathrm{negative ~ curves}} \mathbb{R}_{\geq 0} [E] $$
\end{proposition}

\begin{proposition}\label{contr}
Let $X$ be a Mori dream surface with $\rho \geq 3.$ Then there is a morphism $X \to X_0$ such that $\rho(X_0)=1$ or $\rho(X_0)=2$ and $X_0$ has a finite morphism to $\mathbb{P}^1 \times \mathbb{P}^1.$
\end{proposition}
\begin{proof}
Let $F$ be a facet of $\Eff(X)$ and let $D_F \in \Nef(X)$ is a nonzero vector which is orthogonal to $F$. The nef cone $\Nef(X)$ of $X$ is a rational polyhedral cone generated by semiample divisors. Suppose that $D_F^2 > 0.$ Then a sufficiently large multiple of $D_F$ defines a morphism $X \to X_0$ which contracts these $\rho-1$ negative curves. We know that $X_0$ is normal $\mathbb{Q}$-factorial from \cite{AHL, Okawa, Sakai}. Hence $X_0$ is a combinatorially minimal Mori dream surface with $\rho=1$. Again, we know that $X_0$ is normal $\mathbb{Q}$-factorial from \cite{AHL, Okawa, Sakai}. Therefore we obtain the desired conclusion. \\

Now suppose that all extremal rays of $\Nef(X)$ of the form $D_F$ above are divisors with self-intersection zero. Let $\sigma$ be a face of codimension two and let $F_1, F_2$ be two facets of $\Eff(X)$ whose intersection is $\sigma.$ Let $D_i \in \Nef(X)$ be a nonzero nef divisor which is orthogonal to $F_i.$ Then $D_1, D_2$ are two adjacent extremal rays of $\Nef(X)$ and let $\phi_i : X \to \mathbb{P}^1$ be the morphism obtained from sufficiently large multiple of $D_i.$ Then we have a morphism $\phi=(\phi_1, \phi_2) : X \to \mathbb{P}^1 \times \mathbb{P}^1.$ There are $\rho-2$ negative curves which are orthogonal to $D_1, D_2.$ These negative curves are contracted by $f.$
\end{proof}

From the above discussion, we have another proof of the following theorem.

\begin{theorem}
Let $X$ be a smooth Mori dream surface. Then \\
(1) We can contract negative curves from $X$ to obtain a combinatorial minimal Mori dream surface $X_0$ as follows. 
$$ X=X_n \to X_{n-1} \to \cdots \to X_0. $$
(2) The surface $X_0$ is normal projective $\mathbb{Q}$-factorial variety and the Picard number of $X_0$ is one or two. \\
(3) If the Picard number of $X_0$ is two, then $X_0$ has a finite morphism to $\mathbb{P}^1 \times \mathbb{P}^1.$ \\
(4) If $X$ is a smooth minimal surface then $X$ is the minimal resolution of $X_0.$ \\
(5) Conversely, if $X_0$ is a $\mathbb{Q}$-factorial surface with $\rho=1$ or $\rho=2$ having a finite morphism to $\mathbb{P}^1 \times \mathbb{P}^1,$ $X_0$ is a combinatorially minimal Mori dream surface.
\end{theorem}
\begin{proof}
The assertions (1), (2) and (3) follow from Proposition \ref{Pic=1,2} and Proposition \ref{contr}. The proofs of assertions (4) and (5) are the same as Theorem \ref{contraction:proof1}.
\end{proof}

From the proof of the above theorem, we can see the following.

\begin{remark}
(1) To obtain $X_0$ with $\rho(X_0)=1,$ we need to find a nef divisor $D_F$ which is orthogonal to a facet $F \subset \Eff(X)$ such that $D^2_F > 0.$ \\
(2) To obtain $X_0$ with $\rho(X_0)=2,$ we need to find two adjacent facets $F_1, F_2 \subset \Eff(X)$ such that the corresponding nef divisors $D_{F_1}, D_{F_2}$ satisfy $D^2_{F_1}=D^2_{F_2}=0.$
\end{remark}

We will apply the above observation to construct combinatorially minimal models of examples of Mori dream surfaces of general types found in \cite{Keum-Lee1}.

\subsection{Examples of Mori dream surfaces of general type with $p_g=0$}

It is well-known that weak del Pezzo surfaces are Mori dream surfaces. Among minimal surfaces with Kodaira dimension 0, K3 surfaces and Enriques surfaces having finite automorphism groups are Mori dream surfaces. For surfaces of general type, we provided several examples of Mori dream surfaces in our previous paper \cite{Keum-Lee1}.

\begin{theorem}\cite{Keum-Lee1}\label{Keum-Lee1main}
The following minimal surfaces of general type with $p_g=0$ are Mori dream surfaces and their negative curves are listed as follows. Here  $m(C^2, p_a(C))$ means $m$ copies of $(C^2, p_a(C)).$
\begin{enumerate}
\item Fake projective planes $($all have $K^2=9),$ none
\item Surfaces isogenous to a higher product of unmixed type $($all have $K^2=8),$ none
\item Inoue surfaces with $K^2=7,$ $2(-1,1), (-1,2)$
\item Surfaces with $K^2=7$ constructed by Y. Chen, $(-1,1), (-1,2), (-1,3), (-4,2)$ 
\item Kulikov surfaces with $K^2=6,$ $6(-1,1)$
\item Burniat surfaces with $2 \leq K^2 \leq 6$
\begin{enumerate}
\item $6(-1,1)$ if $K^2 = 6;$
\item $9(-1,1), (-4,0)$ if $K^2 = 5;$
\item $12(-1,1), 4(-4,0)$ if non-nodal with $K^2 = 4;$
\item $10(-1,1), 2(-4,0), (-2,0)$ if nodal with $K^2 = 4;$
\item $9(-1,1), 3(-4,0), 3(-2,0)$ if $K^2 = 3;$
\item $6(-1,1), 6(-2,0), 4(-4,0)$  if $K^2 = 2;$
\end{enumerate}
\item Product-quotient surfaces with $K^2=6, G=D_4 \times \mathbb{Z}/2\mathbb{Z},$ $2(-2, 0), (-1,1), (-1, 2)$
\item A family of Keum-Naie surfaces which are product-quotient surfaces with $K^2=4, G=\mathbb{Z}/4\mathbb{Z} \times \mathbb{Z}/2\mathbb{Z},$ $4(-1, 1), 4(-2, 0)$
\end{enumerate}
\end{theorem}

Recently, Frapporti and the second named author studied effective, nef and semiample cones of surfaces isogenous to a product of mixed type and provided more examples of Mori dream surfaces of general type in \cite{FL}.

\begin{theorem}\cite{FL}
All surfaces isogenous to higher product with $p_g=0$ of mixed type are Mori dream surfaces. \end{theorem}

We refer \cite{FL, Keum-Lee1} and references therein for more details about these surfaces. We will also use the same notations used in \cite{FL, Keum-Lee1}.

\subsection{Fake projective planes and fake quadrics}

We can easily see that all fake projective planes and some fake quadrics are combinatorially minimal Mori dream surfaces.

\begin{proposition}\cite{FL, Keum-Lee1}
Let $X$ be a fake projective plane or a reducible fake quadric. Then $X$ is a combinatorially minimal Mori dream surface.
\end{proposition}
\begin{proof}
The statement is obvious for fake projective planes since they are smooth projective surfaces with Picard number 1. For a surface isogenous to a higher product of unmixed type we can see that the effective cone is the same as the semiample cone from the construction. For surfaces isogenous to a product with $p_g=0$ of mixed type, the result follows explicit descriptions of effective, nef, semiamples cones in \cite{FL}.
\end{proof}

However it is still an open question whether all fake quadrics (especially irreducible fake quadrics) are combinatorially Mori dream surfaces or not.

\begin{question}
Let $X$ be a fake quadric. Is $X$ a Mori dream surface? Is $X$ a combinatorially minimal Mori dream surface?
\end{question}

\subsection{Inoue surfaces with $K^2=7$}

In \cite{Keum-Lee1}, we showed that the effective cone $\Eff(X)$ is generated by three divisors, and $\Nef(X)$ is also generated by three divisors. The three divisors are all negative curves on the Inoue surfaces.

\begin{center}
\begin{tabular}{|c|c|c|c|c|} \hline
$\cdot$ & $\widetilde{\Delta}_1$ & $\widetilde{\Delta}_2$ & $\widetilde{\Delta}_3$ \\
\hline $\widetilde{\Delta}_1$ & -1 & 1 & 1 \\
\hline $\widetilde{\Delta}_2$ & 1 & -1 & 1 \\
\hline $\widetilde{\Delta}_3$ & 1 & 1 & -1 \\
\hline
\end{tabular}
\end{center}

\begin{proposition}
There are three fibrations of $X.$
\end{proposition}
\begin{proof}
The nef cone $\Nef(X)$ is generated by nef divisors $\widetilde{\Delta}_1+\widetilde{\Delta}_2, \widetilde{\Delta}_2+\widetilde{\Delta}_3, \widetilde{\Delta}_3+\widetilde{\Delta}_1$. Because they are all semiample divisors having self-intersection $0$ they induce fibrations. From the configuration of $\Nef(X)$ we see that any semiample divisor having self-intersection $0$ is multiple of one of these three divisors.
\end{proof}

\begin{proposition}
Let $X=X_n \to X_{n-1} \to \cdots \to X_0$ be a combinatorial contractions to a combinatorially minimal Mori dream surface $X_0.$ Then $\rho(X_0) \geq 2$ and there are three ways to obtain combinatorially minimal models.
\end{proposition}
\begin{proof}
We can check that there are no $2 \times 2$ negative definite submatrix of the above intersection matrix. Therefore we see that $\rho(X_0) \geq 2.$ We can contract $\widetilde{\Delta}_i$ $(i=1,2 ~ \mathrm{or} ~ 3)$ to obtain combinatorially minimal Mori dream surfaces with $\rho=2.$
\end{proof}

\subsection{Chen's surfaces with $K^2=7$}

In \cite{Keum-Lee1}, we showed that the effective cone $\Eff(X)$ is generated by four divisors, and $\Nef(X)$ is also generated by four divisors. The four divisors are all negative curves on Chen's surfaces. The intersection matrix is given as follows.

\begin{center}
\begin{tabular}{|c|c|c|c|c|} \hline
$\cdot$ & $\widetilde{E}$ & $\widetilde{\Gamma}$ & $\widetilde{B_2}$ & $\widetilde{B_3}$ \\
\hline $\widetilde{E}$ & -4 & 2 & 2 & 6 \\
\hline $\widetilde{\Gamma}$ & 2 & -1 & 3 & 1 \\
\hline $\widetilde{B_2}$ & 2 & 3 & -1 & 1 \\
\hline $\widetilde{B_3}$ & 6 & 1 & 1 & -1 \\
\hline
\end{tabular}
\end{center}

\begin{proposition}
There are four fibrations of $X.$
\end{proposition}
\begin{proof}
There are exactly four semiample divisors $\widetilde{E}+2\widetilde{\Gamma},$ $\widetilde{E}+2\widetilde{B_2},$ $\widetilde{\Gamma}+\widetilde{B_3},$ $\widetilde{B_2}+\widetilde{B_3}$ with self-intersection $0$ on $X.$ From the configuration of $\Nef(X)$ we see that any semiample divisor having self-intersection $0$ is multiple of one of these four divisors. Therefore we see that there are exactly fibration structures of $X.$
\end{proof}

\begin{proposition}
Let $X=X_n \to X_{n-1} \to \cdots \to X_0$ be a combinatorial contractions to a combinatorially minimal Mori dream surface $X_0.$ Then $\rho(X_0)=2$.
\end{proposition}
\begin{proof}
We can see the pairs $(\widetilde{E}, \widetilde{B_3}), (\widetilde{\Gamma}, \widetilde{B_2})$ form negative definite $2 \times 2$ lattices. However, there is no semiample divisor $D$ whose intersection numbers are zero with these pairs. Therefore we see that there is no combinatorial minimal model with $\rho=1$. \\

Let us suppose that $\rho(X_0) = 2.$ When we choose two fibrations among four induced by $\widetilde{E}+2\widetilde{\Gamma},$ $\widetilde{E}+2\widetilde{B_2},$ $\widetilde{\Gamma}+\widetilde{B_3},$ $\widetilde{B_2}+\widetilde{B_3}$, there is an induced map to $\mathbb{P}^1 \times \mathbb{P}^1.$ For each choice of two (numerical classes of) semiample divisors we can check that there are four curves contracted by the fibrations as follows. \\
(1) Consider the set $\{ \widetilde{E}+2\widetilde{\Gamma}, \widetilde{E}+2\widetilde{B_2} \}.$ Then the induced morphism $X \to \mathbb{P}^1 \times \mathbb{P}^1$ contracts $\widetilde{E}.$ In this case, $X_0$  has one singular point. \\
(2) Consider the set $\{ \widetilde{E}+2\widetilde{\Gamma}, \widetilde{\Gamma}+\widetilde{B_3} \}.$ Then the induced morphism $X \to \mathbb{P}^1 \times \mathbb{P}^1$ contracts $\widetilde{\Gamma}.$ In this case, $X_0$  has one singular point. \\
(3) Consider the set $\{ \widetilde{E}+2\widetilde{B_2}, \widetilde{B_2}+\widetilde{B_3} \}.$ Then the induced morphism $X \to \mathbb{P}^1 \times \mathbb{P}^1$ contracts $\widetilde{B_2}.$ In this case, $X_0$ has one singular point. \\
(4) Consider the set $\{ \widetilde{\Gamma}+\widetilde{B_3}, \widetilde{B_2}+\widetilde{B_3} \}.$ Then the induced morphism $X \to \mathbb{P}^1 \times \mathbb{P}^1$ contracts $\widetilde{B_3}.$ In this case, $X_0$ has one singular point. \\  

Therefore we can see that there are four combinatorially minimal model $X_0$ with $\rho=2.$
\end{proof}

\subsection{Kulikov surfaces with $K^2=6$}

Let $X$ be a Kulikov surface with $K^2=6.$ Recall that there are six negative curves $\widetilde{E_1}, \widetilde{E_2}, \widetilde{E_3},$ $\widetilde{L_1}, \widetilde{L_2}, \widetilde{L_3}$ on $X$ and the intersection matrix among them is as follows.

\begin{center}
\begin{tabular}{|c|c|c|c|c|c|c|} \hline
$\cdot$ & $\widetilde{E_1}$ & $\widetilde{E_2}$ & $\widetilde{E_3}$ & $\widetilde{L_1}$ & $\widetilde{L_2}$ & $\widetilde{L_3}$ \\
\hline $\widetilde{E_1}$ & -1 & 0 & 0 & 0 & 1 & 1 \\
\hline $\widetilde{E_2}$ & 0 & -1 & 0 & 1 & 0 & 1 \\
\hline $\widetilde{E_3}$ & 0 & 0 & -1 & 1 & 1 & 0 \\
\hline $\widetilde{L_1}$ & 0 & 1 & 1 & -1 & 0 & 0 \\
\hline $\widetilde{L_2}$ & 1 & 0 & 1 & 0 & -1 & 0 \\
\hline $\widetilde{L_3}$ & 1 & 1 & 0 & 0 & 0 & -1 \\
\hline
\end{tabular}
\end{center}

\begin{proposition}
There are three fibrations of $X.$
\end{proposition}
\begin{proof}
Because $\Eff(X)$ is generated by the six negative curves, we can compute $\Nef(X)$ explicitly. Among the extremal rays of $\Nef(X)$, three classes $[\widetilde{E_1}+\widetilde{L_2}], [\widetilde{E_2}+\widetilde{L_3}], [\widetilde{E_3}+\widetilde{L_1}]$ are semiample divisor classes with self-intersection $0.$
\end{proof}

\begin{proposition}
Let $X=X_n \to X_{n-1} \to \cdots \to X_0$ be a combinatorial contractions to a combinatorially minimal Mori dream surface $X_0.$ Then $\rho(X_0)=1$ or $2$.
\end{proposition}
\begin{proof}
We can contract $(\widetilde{E_1}, \widetilde{E_2}, \widetilde{E_3})$ or $(\widetilde{L_1}, \widetilde{L_2}, \widetilde{L_3})$ to obtain combinatorially minimal Mori dream surfaces with $\rho=1.$ In this case, $X_0$ has three simple elliptic singular points. We can contract $(\widetilde{E_1}, \widetilde{L_1}), (\widetilde{E_2}, \widetilde{L_2})$ or $(\widetilde{E_3}, \widetilde{L_3})$ to obtain combinatorially minimal Mori dream surfaces with $\rho=2.$ Let $i \in \{1, 2, 3 \}.$ When we contract $(\widetilde{E_i}, \widetilde{L_i}),$ $X_0$ has two fibrations $\mathbb{P}^1$ which is induced by $\widetilde{E_i}+\widetilde{L_{i-1}},$ $\widetilde{E_i}+\widetilde{L_{i+1}}.$
\end{proof}

\subsection{Burniat surfaces with $K^2 = 6$}

Let $X$ be a Burniat surface with $K^2=6.$ Recall that there are six negative curves $\widetilde{E_1}, \widetilde{E_2}, \widetilde{E_3},$ $\widetilde{L_1}, \widetilde{L_2}, \widetilde{L_3}$ on $X$ and the intersection matrix among them is as follows.

\begin{center}
\begin{tabular}{|c|c|c|c|c|c|c|} \hline
$\cdot$ & $\widetilde{E_1}$ & $\widetilde{E_2}$ & $\widetilde{E_3}$ & $\widetilde{L_1}$ & $\widetilde{L_2}$ & $\widetilde{L_3}$ \\
\hline $\widetilde{E_1}$ & -1 & 0 & 0 & 0 & 1 & 1 \\
\hline $\widetilde{E_2}$ & 0 & -1 & 0 & 1 & 0 & 1 \\
\hline $\widetilde{E_3}$ & 0 & 0 & -1 & 1 & 1 & 0 \\
\hline $\widetilde{L_1}$ & 0 & 1 & 1 & -1 & 0 & 0 \\
\hline $\widetilde{L_2}$ & 1 & 0 & 1 & 0 & -1 & 0 \\
\hline $\widetilde{L_3}$ & 1 & 1 & 0 & 0 & 0 & -1 \\
\hline
\end{tabular}
\end{center}

Note that the configuration of negative curves on the primary Burniat surfaces is the same as the Kulikov surfaces cases. Therefore we have the following Propositions.

\begin{proposition}
There are three fibrations of $X.$
\end{proposition}

\begin{proposition}
Let $X=X_n \to X_{n-1} \to \cdots \to X_0$ be a combinatorial contractions to a combinatorially minimal Mori dream surface $X_0.$ Then $\rho(X_0)=1$ or $2$.
\end{proposition}

\subsection{Burniat surfaces with $2 \leq K^2 \leq 5$}

The surface $X$ is a bidouble cover of a weak del Pezzo surface $Y$ such that $\rho(X)=\rho(Y).$ One can find an extremal nef divisor having positive self-intersection on $Y$ which induces a morphism contracting $\rho-1$ negative curves $Y \to Y_0$. Then its pullback is an extremal nef divisor having positive self-intersection on $X.$ And we see this semiample divisor induces a morphism $X \to X_0$ contracting $\rho-1$ negative curves. that there are combinatorially minimal model $X_0$ such that $\rho(X_0)=1.$ Similarly, one can find two fibrations of $Y$ such that the Stein factorization $Y \to Y_0 \to \mathbb{P}^1 \times \mathbb{P}^1$ of the induced morphism $Y \to \mathbb{P}^1 \times \mathbb{P}^1$ contracts $\rho-2$ negative curves. By pulling back the two extremal nef divisors, we can find two fibrations of $X$ and it induces a morphism $X \to X_0$ such that $\rho(X_0)=2.$ Therefore we have the following conclusion.

\begin{proposition}
Let $X$ be a Burniat surface with $2 \leq K^2 \leq 5.$ Let $X=X_n \to X_{n-1} \to \cdots \to X_0$ be a combinatorial contractions to a combinatorially minimal Mori dream surface $X_0.$ Then $\rho(X_0)=1$ or $2$.

\end{proposition}

\subsection{Product-quotient surfaces with $K^2=6, G=D_4 \times \mathbb{Z}/2\mathbb{Z}$}

Recall that there are four negative curves $E_1, E_2, F_1, G_1$ on $X$ and the intersection matrix among them is as follows. See \cite{Keum-Lee1} for more details.

\begin{center}
\begin{tabular}{|c|c|c|c|c|} \hline
$\cdot$ & $E_1$ & $E_2$ & $F_1$ & $G_1$ \\
\hline $E_1$ & -2 & 0 & 1 & 1 \\
\hline $E_2$ & 0 & -2 & 1 & 1\\
\hline $F_1$ & 1 & 1 & -1 & 0 \\
\hline $G_1$ & 1 & 1 & 0 & -1 \\
\hline
\end{tabular}
\end{center}

\begin{proposition}
There are four fibrations of $X.$
\end{proposition}
\begin{proof}
The nef cone $\Nef(X)$ is generated by nef divisors $F_1+E_1+G_1, F_1+E_2+G_1, E_1+E_2+2F_1, E_1+E_2+2G_1.$ Because they are all semiample divisors having self-intersection $0$ they induce fibrations. From the configuration of $\Nef(X)$ we see that any semiample divisor having self-intersection $0$ is multiple of one of these four divisors.
\end{proof}

\begin{proposition}
Let $X=X_n \to X_{n-1} \to \cdots \to X_0$ be a combinatorial contractions to a combinatorially minimal Mori dream surface $X_0.$ Then $\rho(X_0) = 2$ and there are six pairs of two negative curves whose contractions give combinatorially minimal Mori dream surfaces.
\end{proposition}
\begin{proof}
We can check that there are no $3 \times 3$ negative definite submatrix of the above intersection matrix. Therefore we see that $\rho(X_0) = 2.$ We can contract $(E_1, E_2), (F_1, G_1), (E_1, F_1), (E_1, G_1), (E_2, F_1), (E_2, G_1)$ to obtain combinatorially minimal Mori dream surfaces with $\rho=2.$ Note that when we contract $(E_1, E_2)$ we obtain combinatorially minimal Mori dream surface $2A_1$ singularities and when contract $(F_1, G_1)$ we obtain combinatorially minimal Mori dream surface with two simple elliptic singularities.
\end{proof}

\subsection{Product-quotient surfaces with $K^2=4, G=\mathbb{Z}/4\mathbb{Z} \times \mathbb{Z}/2\mathbb{Z}$}

Recall that there are eight negative curves $E_1, E_2, E_3, E_4, F_1, F_2, G_1, G_2$ on $X$ and the intersection matrix among them is as follows. See \cite{Keum-Lee1} for more details.

\begin{center}
\begin{tabular}{|c|c|c|c|c|c|c|c|c|} \hline
$\cdot$ & $E_1$ & $E_2$ & $E_3$ & $E_4$ & $F_1$ & $F_2$ & $G_1$ & $G_2$ \\
\hline $E_1$ & -2 & 0 & 0 & 0 & 1 & 0 & 0 & 1 \\
\hline $E_2$ & 0 & -2 & 0 & 0 & 1 & 0 & 1 & 0 \\
\hline $E_3$ & 0 & 0 & -2 & 0 & 0 & 1 & 1 & 0 \\
\hline $E_4$ & 0 & 0 & 0 & -2 & 0 & 1 & 0 & 1 \\
\hline $F_1$ & 1 & 1 & 0 & 0 & -1 & 0 & 0 & 0 \\
\hline $F_2$ & 0 & 0 & 1 & 1 & 0 & -1 & 0 & 0 \\
\hline $G_1$ & 0 & 1 & 1 & 0 & 0 & 0 & -1 & 0 \\
\hline $G_2$ & 1 & 0 & 0 & 1 & 0 & 0 & 0 & -1 \\
\hline
\end{tabular}
\end{center}

\begin{proposition}
There are four fibrations of $X.$
\end{proposition}
\begin{proof}
In \cite{Keum-Lee1} we computed effective cone and found all its extremal rays. Therefore we can find all extremal rays of nef cone. We can check that a nef divisor with self-intersection $0$ is numerically equivalent to one of $F_1+E_1+G_2, F_1+E_2+G_1, E_1+E_2+2F_1, E_1+E_4+2G_2.$ Because they are all semiample divisors having self-intersection $0$ they induce fibrations. \end{proof}

\begin{proposition}
Let $X=X_n \to X_{n-1} \to \cdots \to X_0$ be a combinatorial contractions to a combinatorially minimal Mori dream surface $X_0.$ Then $\rho(X_0)$ can be 1 or 2.
\end{proposition}
\begin{proof}
In order to find $X_0$ such that $\rho(X_0)=1$, we need to find a $5 \times 5$ negative definite submatrix of the above intersection matrix. Note that the configuration of negative curves $E_1, E_2, E_3, E_4, F_1, F_2, G_1, G_2$ has (numerical) symmetry. Therefore it is enough to check the following 10 configurations. \\
(1-1) Consider the set $\{ F_1, E_1, G_2, E_4, F_2 \}.$ Then $F_1+E_1+G_2$ is a semiample divisor and hence the intersection matrix of $\{ F_1, E_1, G_2, E_4, F_2 \}$ is not negative definite. \\
(1-2) Consider the set $\{ E_1, G_2, E_4, F_2, E_3 \}.$ Then $E_1+2G_2+E_4$ is a semiample divisor and hence the intersection matrix of $\{ E_1, G_2, E_4, F_2, E_3 \}$ is not negative definite. \\
(1-3) Consider the set $\{ E_1, G_2, E_4, F_2, G_1 \}.$ Then $E_1+2G_2+E_4$ is a  semiample divisor and hence the intersection matrix of $\{ E_1, G_2, E_4, F_2, G_1 \}$ is not negative definite. \\
(1-4) Consider the set $\{ E_1, G_2, E_4, E_3, G_1 \}.$ Then $E_1+2G_2+E_4$ is a  semiample divisor and hence the intersection matrix of $\{ E_1, G_2, E_4, E_3, G_1 \}$ is not negative definite. \\
(1-5) Consider the set $\{ G_2, E_4, F_2, G_1, E_2 \}.$ Then $F_2+E_4+G_2$ is a  semiample divisor and hence the intersection matrix of $\{ G_2, E_4, F_2, G_1, E_2 \}$ is not negative definite. \\
(1-6) Consider the set $\{ G_2, E_4, F_2, E_3, E_2 \}.$ Then $F_2+E_4+G_2$ is a  semiample divisor and hence the intersection matrix of $\{ G_2, E_4, F_2, E_3, E_2 \}$ is not negative definite. \\
(1-7) Consider the set $\{ E_1, G_2, E_4, E_3, E_2 \}.$ Then $E_1+2G_2+E_4$ is a  semiample divisor and hence the intersection matrix of $\{ E_1, G_2, E_4, E_3, E_2 \}$ is not negative definite. \\
(1-8) Consider the set $\{ E_1, E_4, F_2, G_1, E_2 \}.$ Then the intersection matrix is as follows. 
$$ (-1) \oplus \left( \begin{array}{cc} -2 & 1 \\ 1 & -1 \end{array} \right) \oplus \left( \begin{array}{cc} -2 & 1 \\ 1 & -1 \end{array} \right) $$ 
Therefore the intersection matrix of $\{ E_1, E_4, F_2, G_1, E_2 \}$ is negative definite. \\
(1-9) Consider the set $\{ E_1, G_2, F_2, G_1, E_2 \}.$ Then the intersection matrix is as follows. 
$$ (-2) \oplus  \left( \begin{array}{cc} -2 & 1 \\ 1 & -1 \end{array} \right) \oplus \left( \begin{array}{cc} -2 & 1 \\ 1 & -1 \end{array} \right) $$ 
Therefore the intersection matrix of $\{ E_1, G_2, F_2, G_1, E_2 \}$ is negative definite. \\
(1-10) Consider the set $\{ G_2, F_2, G_1, E_2, F_1 \}.$ Then $F_1+2E_2+G_1$ is a  semiample divisor and hence the intersection matrix of $\{ G_2, F_2, G_1, E_2, F_1 \}$ is not negative definite. \\

Therefore we see that there is $5 \times 5$ negative definite matrix for the cases (1-8) and (1-9). For the case (1-8), multiples of the semiample divisor $E_1+2G_2+2E_4+2F_2$ induces a morphism $X \to X_0$ contracting $5$ negative curves hence we obtain a combinatorially minimal model with $\rho(X_0)=1.$ For the case (1-9), multiples of the semiample divisor $E_1+2G_2+E_4+F_2$ induces a morphism $X \to X_0$ contracting $5$ negative curves hence we obtain a combinatorially minimal model with $\rho(X_0)=1.$ \\

Let us suppose that $\rho(X_0) = 2.$ When we choose two fibrations among four induced by $F_1+E_1+G_2, F_1+E_2+G_1, E_1+E_2+2F_1, E_1+E_4+2G_2$ there is an induced map to $\mathbb{P}^1 \times \mathbb{P}^1.$ For each choice of two (numerical classes of) semiample divisors we can check that there are four curves contracted by the fibrations as follows. \\
(2-1) Consider the set $\{ F_1+E_1+G_2, F_1+E_2+G_1 \}.$ Then the induced morphism $X \to \mathbb{P}^1 \times \mathbb{P}^1$ contracts $F_1, F_2, G_1, G_2.$ In this case, $X_0$ has four simple elliptic singularities. \\
(2-2) Consider the set $\{ F_1+E_1+G_2, E_1+E_2+2F_1 \}.$ Then the induced morphism $X \to \mathbb{P}^1 \times \mathbb{P}^1$ contracts $E_1, F_1, E_3, F_2.$ In this case, $X_0$ has two singular points. \\
(2-3) Consider the set $\{ F_1+E_1+G_2, E_1+E_4+2G_2 \}.$ Then the induced morphism $X \to \mathbb{P}^1 \times \mathbb{P}^1$ contracts $E_1, G_2, E_3, G_1.$ In this case, $X_0$ has two singular points. \\
(2-4) Consider the set $\{ F_1+E_2+G_1, E_1+E_2+2F_1 \}.$ Then the induced morphism $X \to \mathbb{P}^1 \times \mathbb{P}^1$ contracts $F_1, E_2, F_2, E_4.$ In this case, $X_0$ has two singular points. \\
(2-5) Consider the set $\{ F_1+E_2+G_1, E_1+E_4+2G_2 \}.$ Then the induced morphism $X \to \mathbb{P}^1 \times \mathbb{P}^1$ contracts $E_2, G_1, G_2, E_4.$ In this case, $X_0$ has two singular points. \\
(2-6) Consider the set $\{ E_1+E_2+2F_1, E_1+E_4+2G_2 \}.$ Then the induced morphism $X \to \mathbb{P}^1 \times \mathbb{P}^1$ contracts $E_1, E_2, E_3, E_4.$ In this case, $X_0$ has four $A_1$ singularities. \\

Therefore we can describe two fibrations inducing a morphism $X \to X_0$ with $\rho(X_0)=2$ and singularities of $X_0.$
\end{proof}

\bigskip

\section{Minimal surfaces of general type with $p_g=0$ as minimal resolutions of combinatorially minimal Mori dream surfaces}\label{s5}

There are many surfaces which we do not know whether they are Mori dream surfaces or not. However we can still ask whether they are minimal resolutions of certain combinatorially minimal Mori dream surfaces.

\begin{question}
Let $X$ be a minimal surface of general type with $p_g=0.$ Can we contract negative curves on $X$ to obtain a combinatorially minimal Mori dream surface? In other words, can we realize $X$ as minimal resolution of a combinatorially minimal Mori dream surface $X_0$? 
\end{question}

In this section, we present several examples of minimal surfaces of general type with $p_g=0$ having combinatorially minimal models.

\subsection{Strategy}

Here, we state certain strategy to construct a combinatorially minimal model $X_0$ from a surface $X$. 

\begin{proposition}
Let $X_0$ be a normal complete surface having a finite morphism to $\mathbb{P}^1 \times \mathbb{P}^1.$ Suppose that a resolution of $X_0$ is a regular surface and $\Cl(X_0)_{\mathbb{R}}$ has dimension 2. Then $X_0$ is combinatorially minimal Mori dream surface with $\rho=2.$
\end{proposition}
\begin{proof}
From \cite[Proposition II.6.5]{Hartshorne}, we see that $X_0$ has finitely generated divisor class group $\Cl(X_0).$ Let $D_1$ (resp. $D_2$) be the Cartier divisor of $X_0$ corresponds to the pullback of $\mathcal{O}(1,0)$ (resp. $\mathcal{O}(0,1)$). Then $D_1, D_2$ are semiample divisors and we have $D^2_1=D^2_2=0, D_1 \cdot D_2 > 0.$ Let $D=a_1D_1+a_2D_2$ be a nef divisor. Then we see that $a_1, a_2 \geq 0$ by intersecting $D$ with two effective divisors $D_1$ and $D_2.$ Therefore we have the following inclusions
$$ \mathbb{R}_{\geq 0} \langle D_1, D_2 \rangle \subset \SAmp(X_0) \subset \Mov(X_0) \subset \Nef(X_0) \subset \mathbb{R}_{\geq 0} \langle D_1, D_2 \rangle $$

Let $D=a_1D_1+a_2D_2$ be an effective divisor. By taking intersection with two nef divisors $D_1, D_2$ we have $a_1, a_2 \geq 0$. Therefore we have $\Eff(X_0) = \mathbb{R}_{\geq 0} \langle D_1, D_2 \rangle$ from the following inclusions. 
$$ \mathbb{R}_{\geq 0} \langle D_1, D_2 \rangle  \subset \Eff(X_0) \subset \mathbb{R}_{\geq 0} \langle D_1, D_2 \rangle $$

Therefore, we have $\Eff(X_0) = \Mov(X_0) = \SAmp(X_0).$ From \cite[Theorem 2.5]{AHL}, we obtain the desired result.
\end{proof}

\begin{corollary}
Let $X$ be a smooth projective surface with $q=0$. Let $X \to \mathbb{P}^1 \times \mathbb{P}^1$ be a morphism contracting $\rho-2$ exceptional curves and let $X \to X_0 \to \mathbb{P}^1 \times \mathbb{P}^1$ be the Stein factorization. Then $X_0$ is a combinatorially minimal model of $X$.
\end{corollary}
\begin{proof}
From \cite[Proposition II.6.5]{Hartshorne}, we see that $X_0$ is a normal complete surface with finitely generated divisor class group $\Cl(X_0).$ Moreover, we see that $\dim \Cl(X_0)_{\mathbb{R}}=2.$ And from the construction, we see that $X_0 \to \mathbb{P}^1 \times \mathbb{P}^1$ is finite. Then from the previous proposition, we have the desired result.
\end{proof}

\subsection{Product-quotient surfaces}

Bauer, Catanese, Grunewald and Pignatelli classified product-quotient surfaces with $p_g=q=0$ in \cite{BCGP}. It is easy to see that all of them are minimal resolutions of combinatorially minimal Mori dream surfaces.

\begin{proposition}
Let $X$ be a product-quotient surface with $p_g=q=0.$ We can contract negative curves on $X$ to obtain a combinatorially minimal Mori dream surface.
\end{proposition}
\begin{proof}
Let $C$ and $D$ to be two curves with $G$-action. By definition, $X$ is the minimal resolution of $X_0=(C \times D)/G.$ 

\begin{displaymath}
\xymatrix{
X \ar[rd]  &  C \times D \ar[ld] \ar[d] \ar[rd]  &  \\
C \ar[d]  &  (C \times D)/G \ar[ld] \ar[rd]  & D \ar[d] \\
C/G &  & D/G
}
\end{displaymath}

We can check that $X_0$ is a combinatorially minimal Mori dream surface with $\rho=2$ and there is a finite morphism from $X_0$ to $\mathbb{P}^1 \times \mathbb{P}^1 = C/G \times D/G.$
\end{proof}

\subsection{Minimal resolutions of $\mathbb{Q}$-homology planes}

The first named author studied quotient of fake projective planes by their automorphisms groups of them. From these results, one can construct minimal surfaces of general type with $p_g=q=0, K^2=3$ from fake projective planes in \cite{Keum08}. From the construction, it is obvious to see that these surfaces are minimal resolutions of combinatorially minimal Mori dream surfaces.

\begin{proposition}
Let $X$ be the minimal resolution of a $\mathbb{Q}$-homology plane. We can contract negative curves on $X$ to obtain a combinatorially minimal Mori dream surface.
\end{proposition}

\subsection{Surfaces with $K^2=7$ constructed by Y. Chen and Y. Shin}

Y. Chen and Y. Shin constructed a new 2-dimensional family of minimal surfaces of general type with $p_g=0, K^2=7$ as a bidouble covers of singular rational surfaces with $K^2=-1$ in \cite{CS}. We do not know whether these surfaces are Mori dream surfaces or not. However we can prove that they are minimal resolutions of combinatorially minimal Mori dream surfaces. Let us briefly recall their construction. See \cite{CS} for more details. \\

Let $p_0, p_1, \cdots, p_4, p$ be six distinct points in $\mathbb{P}^2$ satisfying certain conditions (\cite[Section 2]{CS}). For $i=1,2,3,4$, let $p_i'$ be the infinitely near point over $p_i$ which corresponds the line $\overline{p_0p_i}.$ Let $W$ be the blow-up of these ten points of $\mathbb{P}^2.$ Let $l$ be the pullback of the line in $\mathbb{P}^2$ and $e_i$ (resp. $e_0, e$) be the total transform of $p_i$ (resp. $p_0, p$) and $e'_i$ be the $(-1)$-curve corresponds to $p'_i.$ For $i=1,2,3,4,$ let $C_i$ be the strict transform of the line $\overline{p_0p_i}$ and $C'_i$ be the strict transform of the $(-1)$-curve corresponds to $p_i.$ Let $\overline{B_3}$ be the strict transform of the line $\overline{p_0p}.$ The rational surface $W$ admits three fibrations induced by $|l-e_0|,$ $|-2K_W+2e|,$ $|-2K_W+2\overline{B}_3|.$ Chen and Shin proved that the fibration $|-2K_W+2e|$ admits a reducible fiber $\overline{B}_1+e,$ where $\overline{B}_1$ is a smooth elliptic curve such that $\overline{B}_1 \cdot e =1.$ They also proved that the fibration $|-2K_W+2 \overline{B}_3|$ admits a reducible fiber $\overline{B}_2+\overline{B}_3,$ where $\overline{B}_2$ is a smooth elliptic curve such that $\overline{B}_2 \cdot \overline{B}_3 =1.$ \\

Let us define $\Delta_1, \Delta_2, \Delta_3$ as follows.
$$ \Delta_1 := \overline{B}_1+C_1+C'_1+C_2+C'_2 $$ 
$$ \Delta_2 := \overline{B}_2+C_3+C'_3 $$ 
$$ \Delta_3 := \overline{B}_3+C_4+C'_4$$ 

Chen and Shin proved that the above date define a smooth bidouble cover $\overline{\pi} : V \to W$ branched along $\Delta_1+\Delta_2+\Delta_3.$ One can contract eight $(-2)$-curves $C_1, C'_1, C_2, C'_2, C_3, C'_3, C_4, C'_4$ to obtain a eight nodal singular surface $\Sigma.$ For each $i,$ $\overline{\pi}^{-1}(C_i)$ is a disjoint union of two $(-1)$-curves and we can contract these sixteen $(-1)$-curves on $V$ to obtain $X.$ Chen and Shin proved that $X$ is a minimal surface of general type with $p_g=0, K^2=7.$

\begin{displaymath}
\xymatrix{
V \ar[r]^{\epsilon} \ar[d]_{\overline{\pi}}^{} & X \ar[d]_{}^{\pi}  \\
W \ar[r]^{\eta} & \Sigma
}
\end{displaymath}

One can check that $\Sigma$ also admits three fibrations induced from $W.$ Therefore $X$ also admits three fibrations.

\begin{proposition}
Let $X$ be a surface with $K^2=7$ constructed by Y. Chen and Y. Shin. We can contract a negative curve on $X$ to obtain a combinatorially minimal Mori dream surface.
\end{proposition}
\begin{proof}
The pullbacks of two semiample divisors induces a map $X \to \mathbb{P}^1 \times \mathbb{P}^1.$ By considering Stein factorization $X \to X_0 \to \mathbb{P}^1 \times \mathbb{P}^1.$ Therefore we see that there is a morphism contracting $\widetilde{B}_3$ from $X$ to obtain $X_0.$ From our strategy we see that $X_0$ is $\mathbb{Q}$-factorial with $\rho=2$ having a finite morphism to $\mathbb{P}^1 \times \mathbb{P}^1$ and hence it is a combinatorially minimal Mori dream surface.
\end{proof}

There are three commuting involutions $g_1, g_2, g_3$ on $X$ and Chen and Shin proved that the minimal resolution of $X/ \langle g_3 \rangle$ is a numerical Campedelli surface (cf. \cite[Proposition 5.1]{CS}). From the above discussion, we see that the numerical Campedelli surface is the minimal resolution of a combinatorially minimal Mori dream surface.

\begin{proposition}
The numerical Campedelli surface constructed by Chen and Shin (cf. \cite{CS}) is the minimal resolution of a combinatorially minimal Mori dream surface.
\end{proposition}
\begin{proof}
There is a morphism $X/ \langle g_3 \rangle \to \Sigma.$ The proof of the previous Proposition is valid for this surface.
\end{proof}

We have the following natural question.

\begin{question}
Do these surfaces constructed by Chen and Shin have finitely generated Cox rings?
\end{question}

\subsection{Surfaces with $K^2=6$ constructed by M. Inoue and M. Lopes and R. Pardini}

In \cite{Inoue}, M. Inoue constructed minimal surfaces of general type with $K^2=6, p_g=0.$ Later M. Lopes and R. Pardini gave a new description of these surfaces in \cite{MLP04}. Let us recall their construction and see \cite{MLP04} for more details. \\

Let $P_1P_2P_3P_4$ be a quadrilateral in $\mathbb{P}^2.$ Let $P_5$ be the intersection of two lines $\overline{P_1P_2}$ and $\overline{P_3P_4}.$ And let $P_6$ be the intersection of two lines $\overline{P_1P_4}$ and $\overline{P_2P_3}.$ Let $\Sigma \to \mathbb{P}^2$ be the blow-up of $P_1, \cdots, P_6$ and $e_i$ be the exceptional curve over $P_i.$ Let $l$ be the pull back of a line in $\mathbb{P}^2.$ The anticanonical linear system $|-K_{\Sigma}|$ induces a birational morphism to a 4-nodal cubic $V \subset \mathbb{P}^2.$ Let $\Delta_1, \Delta_2, \Delta_3$ be the strict transforms of the lines $\overline{P_1P_3}, \overline{P_2P_4}, \overline{P_5P_6}.$ Let $c_1$ be the strict transform of a general conic through $P_2P_4P_5P_6,$ $c_2$ be the strict transform of a general conic through $P_1P_3P_5P_6,$ $c_3$ be the strict transform of a general conic through $P_1P_2P_3P_4.$ \\

Let $D_1=\Delta_1+f_2+S_1+S_2,$ $D_2=\Delta_2 + f_3,$ $D_3=\Delta_3+f_1+f'_1+S_3+S_4$ where $f_1, f'_1 \in |f_1|, f_2 \in |f_2|, f_3 \in |f_3|$ are general curves. By defining $L_1=5l-e_1-2e_2-e_3-3e_4-2e_5-2e_6,$ $L_2=6l-2e_1-2e_2-2e_3-2e_4-3e_5-3e_6$ and $L_3=4l-2e_1-2e_2-2e_3-e_4-e_5-e_6$ and they give a bidouble covering over $\Sigma.$ \\

Suppose that $f_1, f_2, f_3$ all pass through a general point $P$ and each of them meet at $P$ transversally. Then we have a $(\mathbb{Z}/2)^2$-cover $X_1 \to V$ and the surface $X_1$ has a $\frac{1}{4}(1,1)$-singularity. Let $P'$ be the image of $P$ in $V$ and let $\hat{V} \to V$ be the blow-up of $P'.$ Let $X$ be the normalization of the fiber product $\hat{V} \times_V X_1.$ Lopes and Pardini proved that $X$ is a smooth minimal surface of general type with $p_g=0, K^2=6.$ They also proved that $X$ is different from primary Burniat surfaces or Kulikov surfaces.

\begin{displaymath}
\xymatrix{
X \ar[r] \ar[d] & X_1 \ar[d]  \\
\hat{V} \ar[r] & V
}
\end{displaymath}

\begin{proposition}
Let $X$ be a surface constructed above. We can contract negative curves on $X$ to obtain a combinatorially minimal Mori dream surface.
\end{proposition}
\begin{proof}
The exceptional divisor of $X \to X_1$ is a smooth rational curve with self-intersection $-4.$ Therefore we see that $X_1$ is a $\mathbb{Q}$-factorial surface with $\rho=3$ (cf. \cite[Proposition 4.4.10]{KM}). On $X_1,$ there are two fibrations induced by two semiample divisors $\widetilde{\Delta}_{i}+\widetilde{\Delta}_{i+1}$ and $\widetilde{\Delta}_{i+1}+\widetilde{\Delta}_{i+2}$ for any $i.$ Therefore we have a morphism $X_1 \to \mathbb{P}^1 \times \mathbb{P}^1$ and it contracts $\widetilde{\Delta}_{i+1}.$ By taking Stein factorization, we see that the morphism factors through $X_1 \to X_0 \to \mathbb{P}^1 \times \mathbb{P}^1.$ Because $X_0$ is $\mathbb{Q}$-factorial surface with $\rho=2$ having a finite morphism to $\mathbb{P}^1 \times \mathbb{P}^1$ and hence a combinatorially minimal Mori dream surface. Because $X$ is a minimal surface of general type, we see that $X$ is the minimal resolution of $X_0.$
\end{proof}

\begin{question}
Is $X$ a Mori dream surface?
\end{question}

It will be an interesting task to determine whether some of the classical minimal surfaces of general type can be realized as minimal resolutions of combinatorially minimal Mori dream surfaces.

\bigskip

\section{Singularities of combinatorially minimal models with \\ Picard number 1 having at worst quotient singularities}\label{s6}

In this and the next section we study possible baskets of quotient singularities of combinatorial minimal models of smooth minimal surfaces of general type with $p_g=0$. Let $X$ be a smooth minimal surface of general type with $p_g=0$ and $X_0$ be a combinatorial minimal model of $X$. In this and the next section, we will assume that $X_0$ has at worst quotient singularities. When the Picard number of $X_0$ is 1, then $X_0$ is $\mathbb{Q}$-homology plane and the maximal number of singular points $X_0$ can have was studied in \cite{HK}. Moreover, when $X_0$ has only Gorenstein singularities, its possible baskets of singularities were studied in \cite{HKO}. In this section, we first show that there are finite number of possible types of singularities $X_0$ can have (cf. Proposition \ref{finite number of types}). Then we discuss possible types of quotient singularities of $X_0$ whose Picard number is 1 when it has the maximum number of singular points. Our main references for this section are \cite{HK, HKO} and sometimes we used Sage (cf. \cite{Sage}) to obtain possible types of singularities.

\subsection{Quotient singularities of surfaces}

Let us review basic results about quotient singularities of surfaces. Our first tool to reduce many singularity types is the orbifold Bogomolov-Miyaoka-Yau inequality. Let us recall its statement.

\begin{theorem}[The orbifold Bogomolov-Miyaoka-Yau inequality]\cite{HK}
Let $X_0$ be a normal projective surface with quotient singularities whose canonical divisor $K_{X_0}$ is nef. Then the following inequality holds
$$ K_{X_0}^2 \leq 3 e_{orb}(X_0), $$
where $e_{orb}(X_0) = e(X_0) - \sum_{x \in Sing(X_0)}(1-\frac{1}{|G_x|})$ and $G_x$ is local fundamental group of the singular point $x.$
\end{theorem}

We also have the following weaker inequality.

\begin{theorem}\cite{HK}
Let $X_0$ be a normal projective surface with quotient singularities whose canonical divisor $K_{X_0}$ (or $-K_{X_0}$) is nef. Then the following inequality holds.
$$ 0 \leq 3 e_{orb}(X_0) $$
\end{theorem}

We can compute $K_0^2$ using the following Lemma. Let $f : X \to X_0$ be the minimal resolution of $X_0.$ Then we have 
$$ K_{X} = f^*K_{X_0} - \sum_{x \in Sing(X_0)}D_x. $$
where $D_x=\sum a_iE_i$ is the effective $\mathbb{Q}$-divisor supported on $f^{-1}(x)$ uniquely determined by the equations $f^* K_{X_0} \cdot E_i = 0$ for all $i$. It is well-known that $0 \leq a_i < 1$ for all $i$. Note that
$$ K_{X_0}^2 = K_{X}^2 - \sum_{x \in Sing(X_0)}D_x^2. $$

Let $R_x$ be the sublattice of $H^2(X,\mathbb{Z})/tors$ generated by the exceptional curves over $x$ and let $R=\bigoplus_{x \in Sing(X_0)}R_x \leq H^2(X,\mathbb{Z})/tors.$ A strong restriction of the possible types of singular points $X_0$ can have is the following Lemma in \cite{HK}.

\begin{lemma}\cite[Lemma 3.3]{HK} We have the followings. \\
(1) $det(R + \langle K_X \rangle)=det(R)K_{X_0}^2$ if $K_{X_0}^2 \neq 0.$ \\
(2) If $X_0$ is a $\mathbb{Q}$-homology plane with quotient singularities, then $|det(R + \langle K_X \rangle)|$ is a square number.
\end{lemma}

Brieskorn classified finite subgroups of $GL(2, \mathbb{C})$ without quasi-reflections in \cite{Brieskorn}. See Table 1 in \cite{HK}. We will use the following notation which was used in \cite{HK}.
$$
\begin{array}{lcl}
 < q,q_1> &:=& \text{ the dual graph of the singularity of type }
 \dfrac{1}{q}(1,q_1)
,\\
< b;s_1, t_1; s_2, t_2;s_3, t_3> &:=& \text{ the tree of the
form} \\ 
 & & \\
 & &
\begin{picture}(40,30)
\put(40,25){$<s_2, t_2>$}
\put(52,15){\line(0,1){6}}
\put(0,5){$<s_1,t_1> -\underset{-b}\circ - <s_3, t_3>$}
\end{picture}
\end{array}
$$

For each singular point $x$, one can compute $D^2_x$ from the following Lemmas. 

\begin{lemma}\cite[Lemma 3.6]{HK}\label{numcycquot}
Let $x$ be a cyclic quotient singular point of $X_0$ of type $[m_1,\cdots,m_l],$ i.e. $f^{-1}(x)$ has $l$ components $E_1,\cdots,E_l$ with $E_i^2=-m_i$ and they form a string. Then we have \\
(1) If $l = 1,$ then $D_x^2=-\frac{(m_1-2)^2}{m_1}$; \\
(2) If $l \geq 2,$ then $D_x^2=2l-\sum m_i + a_1 + a_l = 2l-\sum m_i + 2 - \frac{q_1+q_l+2}{q}$; \\
where $q_i=[m_1,\cdots,\widehat{m}_i,\cdots,m_l]$.
\end{lemma}

\begin{lemma}\cite[Lemma 3.7]{HK}\label{non-cycquot}
Let $x$ be a non-cyclic quotient singular point of $X_0$ of type $D_{q,q_1}$ with the dual graph $\langle b; 2,1; 2,1; q,q_1 \rangle$. Let $l$ be the length of the string $\langle q, q_1 \rangle$ and let us assume that $l \geq 2$. Then we have \\
(1) $\mathrm{det}(R_x) = (-1)^{l+3} 4((b-1)q-q_1)$; \\
(2) $a_l = 1 - \frac{(b-1)q_l-q_{1,l}}{(b-1)q-q_1}$; \\
(3) $D_x^2=2l-\sum m_i + a_l -(b-2).$
\end{lemma}

The configuration of exceptional curves over the singular point of type $D_{q,q_1}$ is as follows.

$$
\begin{array}{lcl}
\begin{picture}(40,30)
\put(28,25){$E_{l+1}$}
\put(33,15){\line(0,1){6}}
\put(0,5){$E_{l+2} -E_0 -E_1 -  \cdots - E_l$}
\end{picture}
\end{array}
$$

We also need to compute $D^2_x$ when $x$ is of type $D_{q,q_1}$ and $l=1.$ 

\begin{lemma}\label{non-cycquot2}
Let $x$ be a non-cyclic quotient singular point of $X_0$ of type $D_{q,q_1}$ with the dual graph $\langle b; 2,1; 2,1; q,q_1 \rangle$. Let $l$ be the length of the string $\langle q, q_1 \rangle$ and let us assume that $l =1$. Then we have \\
(1) $\mathrm{det}(R_x) = 4m(b-1-\frac{1}{m})$; \\
(2) $a_0 = \frac{mb-m-2}{mb-m-1}$, $a_1 = \frac{mb-m-2}{2(mb-m-1)}$, $a_2 = \frac{mb-m-2}{2(mb-m-1)}$, $a_3 = \frac{mb-m-b}{mb-m-1}$; \\
(3) $D_x^2 = -ba_0^2-2a_1^2-2a_2^2-ma_3^2+2a_0(a_1+a_2+a_3).$
\end{lemma}
\begin{proof}
Let
$$
\begin{array}{lcl}
\begin{picture}(40,30)
\put(22,25){$E_2$}
\put(27,15){\line(0,1){6}}
\put(0,5){$E_1 -E_0 -E_3$}
\end{picture}
\end{array}
$$
be the configuration of $f^{-1}(x)$ where $E^2_0=-b$, $E^2_1=-2$, $E^2_2=-2$, $E^2_3=-m$. \\
(1) It is easy to compute $\mathrm{det}(R_x) = 4m(b-1-\frac{1}{m})$. \\
(2) Let $D_x = a_0E_0 + a_1E_1 + a_2E_2 + a_3E_3$. From $K_X \cdot E_i$, we have the following matrix equation. 
$$ 
\begin{bmatrix}
    b & -1 & -1 & -1 \\
    -1 & 2 & 0 & 0 \\
    -1 & 0 & 2 & 0 \\
    -1 & 0 & 0 & m
\end{bmatrix}
\begin{bmatrix}
    a_0 \\
    a_1 \\
    a_2 \\
    a_3
\end{bmatrix} 
= 
\begin{bmatrix}
    b-2 \\
    0 \\
    0 \\
    m-2
\end{bmatrix} 
$$
Then we have $a_0 = \frac{mb-m-2}{mb-m-1}$, $a_1 = \frac{mb-m-2}{2(mb-m-1)}$, $a_2 = \frac{mb-m-2}{2(mb-m-1)}$, $a_3 = \frac{mb-m-b}{mb-m-1}$, \\
(3) From the above configuration, we have $D_x^2 = (a_0E_0 + a_1E_1 + a_2E_2 + a_3E_3)^2=-ba_0^2-2a_1^2-2a_2^2-ma_3^2+2a_0(a_1+a_2+a_3).$ 
\end{proof}

Our next tool to reduce many types of quotient singular points is the theory of quadratic forms. Let us briefly recall several definitions and results. 

\begin{definition}\cite[4. Definition 7]{Serre}
Two quadratic form $Q$ and $Q'$ are equivalent(we will write $Q \sim Q'$) if the corresponding modules are equivalent. If $A$ and $A'$ be the corresponding matrix, then $Q \sim Q'$ if there exists an invertible matrix $B$ such that $A'=BAB^t$.
\end{definition}

\begin{theorem}\cite{Serre}
Let $Q$ and $Q'$ be two quadratic forms over $\mathbb{Q}.$ The two quadratic forms $Q$ and $Q'$ are equivalent over $\mathbb{Q}$ if and only if they are equivalent over $\mathbb{R}$ and $\mathbb{Q}_p$ for any prime number $p.$
\end{theorem}

An important invariant of quadratic forms defined over $\mathbb{R}$ or $\mathbb{Q}_p$ is the Hilbert symbol defined as follows.

\begin{definition}\cite{Serre}
Let $k$ be $\mathbb{R}$ or $\mathbb{Q}_p$ where $p$ is a prime number. For $a, b \in k^*$ we define the Hilbert symbol $(a,b)$ of $a$ and $b$ relative to $k$ as follows. \\
(1) $(a,b)=1$ if $z^2-ax^2-by^2=0$ has a solution if $(z,x,y) \neq (0,0,0) \in k^3.$ \\
(2) $(a,b)=-1$ otherwise.
\end{definition}

Then we can compute the Hilbert symbols as follows.

\begin{theorem}\cite{Serre}
(1) If $k=\mathbb{R},$ then $(a,b)=1$ if $a$ or $b$ is $>0$ and $(a,b)=-1$ if $a$ and $b$ are $<0.$  \\
(2) Let $k=\mathbb{Q}_p$ where $p$ is a prime number. For $a=p^{\alpha}a', b=p^{\beta}b' \in k^*$ where $a',b'$ are $p$-adic units, then the Hilbert symbol $(a,b)$ is follows. \\
(2-1) $(a,b)_p=(-1)^{\alpha \beta \epsilon(p)}(\frac{a'}{p})^{\beta}(\frac{b'}{p})^{\alpha}$ if $p \neq 2$ \\
(2-2) $(a,b)_p=(-1)^{\epsilon(a') \epsilon(b')+\alpha \omega(b')+\beta \omega(a')}$ if $p=2$ \\
where $(\frac{c}{p})$ is the Legendre symbol, $\epsilon(c)$ is the class of $\frac{c-1}{2}$ modulo 2 and $\omega(c)$ is the class of $\frac{c^2-1}{8}$ modulo 2.
\end{theorem}

Now we define two important invariants of quadratic forms over $\mathbb{Q}_p.$

\begin{definition}
Let $Q=a_1X^2_1+ \cdots +a_nX^2_n$ be a quadratic form over $\mathbb{Q}_p.$   \\
(1) The discriminant of $Q$ is defined as follows.
$$ d_p(Q) = a_1 ~ \cdots ~ a_n \in \mathbb{Q}_p/\mathbb{Q}^{*2}_p.$$
(2) The $\epsilon$-invariant of $Q$ is defined as follows.
$$ {\epsilon}_p(Q) = \prod_{i < j} (a_i, a_j)_p $$
\end{definition}

We can compute discriminant and $\epsilon$-invariant of a quadratic using the following properties. See \cite{HK} for more details.

\begin{remark}\cite{HK}
Let $Q$ and $Q'$ be two quadratic forms over $\mathbb{Q}_p.$ Then we have the following properties.
$$ d_p(Q \oplus Q') = d_p(Q)d_p(Q') $$
$$ {\epsilon}_p(Q \oplus Q') = {\epsilon}_p(Q) {\epsilon}_p(Q') (d_p(Q), d_p(Q'))_p $$
\end{remark}

\begin{theorem}\cite{Serre}
Let $Q$ and $Q'$ be two quadratic forms over $\mathbb{Q}_p.$ The two quadratic forms $Q$ and $Q'$ are equivalent over $\mathbb{Q}_p$ if and only if they have the same rank, discriminant, $\epsilon$-invariant.
\end{theorem}

Using the following Lemma, we can compute these invariants of lattices.

\begin{lemma}\cite[Lemma 6.4]{HK} Let $L$ be the integral lattice obtained from the Hirzebruch-Jung continued fraction $[n_1, n_2, \cdots, n_l]$ with standard basis. Let $(L \otimes \mathbb{Q}, Q)$ be the quadratic form over $\mathbb{Q}$ defined by $L.$ Then there is an orthogonal basis $\{ v_1, \cdots, v_l \}$ with $v_i^2=-[n_i, \cdots, n_1]$ such that the quadratic form is given by $Q=\sum_i v_i^2X^2_i.$
\end{lemma}

\begin{lemma}
(1) Let $I_{1,\rho-1}$ be the odd unimodular lattice of signature $(1,\rho-1).$ Then $\epsilon_2(I_{1,\rho-1})=(-1)^{\rho}.$ \\
(2) \cite[Lemma 6.6]{HK} Let $I_{1,\rho-1}$ be the odd unimodular lattice of signature $(1,\rho-1).$ Then $\epsilon_p(I_{1,\rho-1})=1$ for any $p > 2.$
\end{lemma}
\begin{proof}
It follows from a direct computation.
\end{proof}

Using these tools, possible types of singular points on $\mathbb{Q}$-homology plane was studied in \cite{HK,HKO}. Let us recall the main theorem of \cite{HK}.

\begin{theorem}\cite{HK}
Let $X_0$ be a $\mathbb{Q}$-homology plane with quotient singularities. Then the number of singular points is less than or equal to 5. Moreover if $X_0$ has five singular points then its minimal resolution is an Enriques surface.  
\end{theorem}

Therefore in our case the number of quotient singular points of $X_0$ is less than or equal to 4. Moreover when $X_0$ has only ADE singularities, a list of possible type of singularities was obtained in \cite{HKO}. In this case, the minimal resolution of $X_0$ is a minimal surface of general type.

\begin{proposition}\label{finite number of types}
Let $X$ which is a smooth minimal surfaces of general type with $p_g=0$ and $X_0$ be a combinatorially minimal model of $X.$ Suppose that $X_0$ has at worst quotient singularities and $\rho(X_0)$ is one or two. Then the number of possible numerical types of singular points of $X_0$ is finite.
\end{proposition}
\begin{proof}
First, let us assume that $X_0$ has only cyclic quotient singularities. From the Bogomolov-Miyaoka-Yau inequality we have $1 \leq K_X^2 \leq 9$ and $1 \leq \rho(X) \leq 9.$ Therefore the number of singular points of $X_0$ is at most eight and number of components of the exceptional curves is also at most 8. From the orbifold Bogomolov-Miyaoka-Yau inequality, we have the following rough inequality.
$$ 1 - \sum_{x \in Sing(X_0)}D_x^2  \leq K_{X_0}^2 = K_{X}^2 - \sum_{x \in Sing(X_0)}D_x^2  \leq $$
$$ 3 e_{orb}(X_0) = 3 e(X_0) - \sum_{x \in Sing(X_0)} 3(1-\frac{1}{|G_x|}) = 12 - \sum_{x \in Sing(X_0)} 3(1-\frac{1}{|G_x|}) \leq 12 $$

For each singular point, we have the following inequalities
$$ 1 + \frac{(m_1-2)^2}{m_1} \leq 1 - \sum_{x \in Sing(X_0)}D_x^2 \leq 12 $$
when $l=1$ and
$$  \sum m_i -1 < 1 - (2l-\sum m_i + a_1 + a_l) \leq 1 - \sum_{x \in Sing(X_0)}D_x^2 \leq 12 $$
when $l \geq 2$ from Lemma \ref{numcycquot}. Therefore we can see that the possible list of $[m_1,\cdots.m_l]$ is finite since $\sum m_i$ is bounded for each singular point. \\

Now let us assume that $X_0$ has at least one non-cyclic quotient singularity. Let $x$ be of type $D_{q,q_1}$. When $l \geq 2$, we have the following inequality from Lemma \ref{non-cycquot}.
$$ 1-10+\sum m_i - 1 +(b-2) < K^2_X-D_x^2=K^2_X-2l+\sum m_i - a_l +(b-2) \leq 12 $$
When $l=1$, we have the following inequality from Lemma \ref{non-cycquot2}.
$$ 1 + \frac{m+b}{4} -4 < K^2_X-D_x^2 = K^2_X - (a_0E_0 + a_1E_1 + a_2E_2 + a_3E_3)^2 $$
$$ = K^2_X + ba_0^2+2a_1^2+2a_2^2+ma_3^2-2a_0(a_1+a_2+a_3) \leq 12 $$
unless $m=b=2$ or $m=2, b=3$. Therefore we see that the number of possible pairs of $(m, b)$ is finite and hence the number of types of the singular point $x$ can be is also finite. \\

Let $x$ be of type $T_m, O_m$ or $I_m$. From Table 2 of \cite{HK}, we see that $K^2_{X_0}$ linearly increase with respect to $b$. However, the orbifold Bogomolov-Miyaoka-Yau inequaltiy implies that $K^2_{X_0}$ is bounded by 12. Therefore we see that $b$ is also bounded. \\ 

From the above discussions, we obtain the desired result.
\end{proof}

 \subsection{Quotient singularities of combinatorially minimal models with Picard number 1}

Now let us study possible types of quotient singular points (including ADE) of the combinatorially minimal Mori dream surfaces with Picard number 1. Let $X$ be a smooth minimal Mori dream surface of general type with $p_g=0$ and let $X_0$ be its combinatorial minimal model of $X$ whose Picard number is 1. Our strategy to study possible singularities $X_0$ can have is as follows. \\
(1) Fix the number of singular points of $X_0$ and $K_X^2$(hence also $\rho(X)$). \\
(2) From the orbifold Bogomolov-Miyaoka-Yau inequality we have finite number of list of possible orders of local fundamental groups. \\
(3) We can compute $K_{X_0}^2$ and check whether $|det(R + \langle K_X \rangle)|$ is a square number or not. \\
(4) Compute Hilbert symbol to check whether $R + \langle K_X \rangle$ can be embedded into a unimodular lattice of signature $(1,\rho(X)-1).$ \\
(5) Use algebro-geometric argument to remove the case or try to find a supporting example. \\
Unfortunately, we could not complete the last step and hence do not know whether our list is too large or not. We left this analysis for future research. \\ 

From now on, we will consider the following situation and fix notation during this section: \\
 
{\bf Situation:} Let $X$ be a smooth minimal surfaces of general type with $p_g=0$ and let $X \to \cdots \to X_0$ be a combinatorial contraction where $X_0$ is a combinatorially minimal Mori dream surface with Picard number 1. Suppose that $X_0$ has at worst quotient singularities.

\begin{lemma}
The number of non-cyclic singular points of $X_0$ is at most two. When the number of singular points of $X_0$ is greater than or equal to two, then $X_0$ has at most one non-cyclic quotient singularity.
\end{lemma}
\begin{proof}
Let $X$ be the minimal resolution of $X_0$ and let $x$ be a non-cyclic singular point of $X_0$. Then number of exceptional components over $x$ is at least four. Note that the difference of the Picard number of $X$ and the Picard number of $X_0$ is less than or equal to eight. Therefore we have the desired result.
\end{proof}

From \cite{HKO}, we see that the number of singular points of $X_0$ is less than 5. When $X_0$ has four singular points, we have the following.

\begin{proposition}
If the number of singular points is 4 then the possible type of singularities of $X_0$ is one of the following list. \\
ADE types : $2A_1 \oplus 2A_3, 4A_2.$ \\
other types : none.
\end{proposition}
\begin{proof}
Let $X_0$ be a surface with $\rho(X_0)=1$ having four quotient singularities and let $X \to X_0$ be its minimal resolution. Suppose that $X_0$ has only cyclic quotient singularities. From the orbifold Bogomolov-Miyaoka-Yau inequality we obtain the following inequality
$$ \frac{1}{3}K_{X}^2 \leq \frac{1}{3}K_{X_0}^2 \leq e_{orb}(X_0) = 3 - \sum_{x \in Sing(X_0)}(1-\frac{1}{|G_x|})=\frac{1}{n_1}+\frac{1}{n_2}+\frac{1}{n_3}+\frac{1}{n_4}-1. $$
Then we have $K_X^2 \leq 3.$ With out loss of generality, we may assume that $n_1 \leq \cdots \leq n_4.$ \\

If $K_X^2=3$ and we see that $n_1=\cdots=n_4=2$ in order to satisfy the above inequality. However the Picard number of $X$ obtained from Noether formula is $\rho(X)=7$ and from resolution, the Picard number is $\rho(X)=5$ hence do not match. Therefore $K_X^2 \neq 3.$ \\ 

If $K_X^2=2$ then $\rho(X)=8$ and we see that $n_1=n_2=2, n_3=n_4=3$ or $n_1=n_2=n_3=2, n_4 \leq 6$ in order to satisfy the above inequality. The first case cannot occur since the Picard number of minimal resolution is less than or equal to 7. The cases $(n_1, n_2, n_3, n_4)=(2,2,2,2), (2,2,2,3), (2,2,2,4), (2,2,2,5), (2,2,2,6)$ cannot occur from the Picard number computation except the singular point is of $3A_1 \oplus A_4$ type. But we can remove this case using the result of \cite{HKO}. \\ 

Therefore we may assume that $K_X^2=1$ and $\rho(X)=9.$ From the orbifold Bogomolov-Miyaoka-Yau inequality we have $\frac{4}{3} \leq \sum_{i=1}^4 \frac{1}{n_i}.$
If $n_1 \geq 3$ then the above inequality holds only if $n_1=\cdots=n_4=3$ and $K_X^2=K_{X_0}^2=1.$ This happen only if $X_0$ has $4A_2$ singularities. And there is a $X_0$ having $4A_2$ singularities (cf. \cite{HKO}). \\

Now suppose that $n_1=2$ and $n_2 \geq 3.$ Then the possible type of $(n_1,n_2,n_3,n_4)$ is $(2,3,3,n),$ where $n \leq 6$ or $(2,3,4,4).$ We can remove the $(2,3,3,3)$ case from the Picard number computation. Again, by considering Picard number, we see that $(2,3,3,4)$ case happen only if all singular points are Gorenstein. Then we can remove this case from \cite{HKO}. For the $(2,3,3,5)$ case, the singular points should be $A_1 \oplus A_2 \oplus \frac{1}{3}(1,1) \oplus A_4$ in order to match the Picard number. However this contradicts to the orbifold Bogomolov-Miyaoka-Yau inequality. For the $(2,3,3,6)$ case, the singular points should be $A_1 \oplus 2 \times \frac{1}{3}(1,1) \oplus A_5$ in order to match the Picard number. However this contradicts to the orbifold Bogomolov-Miyaoka-Yau inequality. For the $(2,3,4,4)$ case, the singular points should be $A_1 \oplus \frac{1}{3}(1,1) \oplus 2A_3$ in order to match the Picard number. However this contradicts to the orbifold Bogomolov-Miyaoka-Yau inequality. Therefore we can see that all these cases cannot occur. \\

Now suppose that $n_1=n_2=2$ and $n_3 \geq 3.$ Then the possible type of $(n_1,n_2,n_3,n_4)$ is $(2,2,3,n)$ or $(2,2,4,n)$ where $n \leq 12$ or $(2,2,5,n)$ where $n \leq 7$ or $(2,2,6,6).$ The last case cannot occur from the computation of Picard number and theorem of \cite{HKO}. Let us consider $n_1=n_2=2,  n_3=3$ case. We can divide this case into two cases. Suppose that the third singular point is of type $A_2.$ Then there are four curves contracted to a singular point. Then we have the following inequality
$$ \frac{1}{3}(1+m_1+m_2+m_3+m_4-10+\frac{m_1m_2m_3-m_1-m_3+m_2m_3m_4-m_2-m_4+2}{m_1m_2m_3m_4-m_1m_2-m_1m_4-m_3m_4+1}) $$
$$ = \frac{1}{3}K_{X_0}^2 \leq \frac{1}{3} + \frac{1}{m_1m_2m_3m_4-m_1m_2-m_1m_4-m_3m_4+1} $$
We can check that the possible type of $(m_1, m_2, m_3, m_4)$ is $(2,2,2,2).$ This case cannot occur from the main theorem of \cite{HKO}. \\

Now suppose that the third singular point is of type $\frac{1}{3}(1,1).$ Then there are five curves contracted to a singular point. Then we have the following inequality
$$ \frac{1}{3}(1+\frac{1}{3}+m_1+m_2+m_3+m_4+m_5-12+\frac{q_1+q_5+2}{q}) = \frac{1}{3}K_{X_0}^2 \leq \frac{1}{3} + \frac{1}{q} $$ 
where 
$$ q=m_1m_2m_3m_4m_5 - m_1m_2m_3 - m_1m_2m_5 - m_1m_4m_5 - m_3m_4m_5 + m_1 + m_3 + m_5, $$
$$ q_ 1= m_1m_2m_3m_4-m_1m_2-m_1m_4-m_3m_4+1, $$
$$ q_5 = m_2m_3m_4m_5-m_2m_3-m_2m_5-m_4m_5+1, $$
and we can check that the possible type of $(m_1, m_2, m_3, m_4, m_5)$ is $(2,2,2,2,2).$ However we have $det(R)K^2_{X_0}=6 \cdot 3 \cdot 2 \cdot 2 \cdot \frac{4}{3}=96$ hence is not a square number. Therefore this case cannot occur. \\

Let us consider $n_1=n_2=2, n_3=4$ case. From orbifold Bogomolov-Miyaoka-Yau inequality, we see that the fourth singular point cannot be of type $\frac{1}{4}(1,1).$ Suppose that the third singular point is of type $A_3.$ Then there are three curves contracted to a singular point. Then we have the following inequality
$$ \frac{1}{3}(1+m_1+m_2+m_3-8+\frac{m_1m_2+m_2m_3}{m_1m_2m_3-m_1-m_3}) $$
$$ = \frac{1}{3}K_{X_0}^2 \leq \frac{1}{4} + \frac{1}{m_1m_2m_3-m_1-m_3} $$
We can check that the possible type of $(m_1, m_2, m_3)$ is $(2,2,2).$ The corresponding singularity is of type $2A_1 \oplus 2A_3.$ And there is a $X_0$ having $2A_1 \oplus 2A_3$ singularities (cf. \cite{HKO}). \\

Let us consider $n_1=n_2=2, n_3=5$ case. From orbifold Bogomolov-Miyaoka-Yau inequality, we see that the third singular point cannot be of type $\frac{1}{5}(1,1)$ or $\frac{1}{5}(1,2).$ Suppose that the third singular point is of type $A_4.$ Then there are two curves contracted to a singular point. Then we have the following inequality
$$ \frac{1}{3}(1+m_1+m_2-6+\frac{m_1+m_2+2}{m_1m_2-1}) = \frac{1}{3}K_{X_0}^2 \leq \frac{1}{5} + \frac{1}{m_1m_2-1} $$
Because $5 \leq m_1m_2-1 \leq 7,$ we can check that there is no pair $(m
_1, m_2)$ satisfying the above property. \\

Now suppose that $n_1=n_2=n_3=2.$ Because $K_X^2=1$ we have $\rho(X)=9.$ Therefore there are five minus curves contracted to a singular point. Then we have the following inequality
$$ \frac{1}{3}(1+m_1+m_2+m_3+m_4+m_5-12+\frac{q_1+q_l+2}{q}) = \frac{1}{3}K_{X_0}^2 \leq \frac{1}{2} + \frac{1}{q} $$ 
where 
$$ q=m_1m_2m_3m_4m_5 - m_1m_2m_3 - m_1m_2m_5 - m_1m_4m_5 - m_3m_4m_5 + m_1 + m_3 + m_5, $$
$$ q_ 1= m_1m_2m_3m_4-m_1m_2-m_1m_4-m_3m_4+1, $$
$$ q_5 = m_2m_3m_4m_5-m_2m_3-m_2m_5-m_4m_5+1, $$
and we see that the possible 5-tuple $(m_1,m_2,m_3,m_4,m_5)$ is one of $(2,2,2,2,2),$ $(3,2,2,2,2),$ $(2,3,2,2,2),$ $(2,2,3,2,2),$ $(2,2,2,3,2),$ $(2,2,2,2,3).$ We can see that the first case cannot occur from \cite{HKO}. For later cases, we can see that $det(R)K^2_{X_0}$ are not square numbers and hence these cases cannot occur. \\

Therefore we see that the only possible cases are $2A_1 \oplus 2A_3$ and $4A_2$ when $X_0$ has only cyclic quotient singularities.
\end{proof}

\begin{remark}
The $2A_1 \oplus 2A_3$ and $4A_2$ cases are supported by examples (cf. \cite{HKO}). 
\end{remark}

We can show that these two cases are the only possible cases. 

\begin{proposition}
Suppose that $X_0$ has four quotient singular points. Then $X_0$ cannot have a non-cyclic quotient singular point.
\end{proposition}
\begin{proof}
Suppose that $X_0$ has a non-cyclic quotient singular point. Then the number of exceptional curves over the singular point is less than or equal to five. Then the possible dual graph of exceptional curves over the non-cyclic singular point is one of $< b; 2,1 ; 2,1 ; q, q_1 >$, $< b; 2,1 ; 3,1 ; 3, 1 >$, $< b; 2,1 ; 3,1 ; 3, 2 >$, $< b; 2,1 ; 3,2 ; 4,1 >$, $< b; 2,1 ; 3,1 ; 4,1 >$, $< b; 2,1 ; 3,1 ; 5,3 >$, $< b; 2,1 ; 3,2 ; 5,1 >$, $< b; 2,1 ; 3,1 ; 5,2 >$, $< b; 2,1 ; 3,1 ; 5,1 >$. Suppose that the singular point is of type $T_m, O_m$ or $I_m$. Then the order of the local fundamental group is $24m, 48m$ or $120m$. We can remove all these cases using the orbifold Bogomolov-Miyaoka-Yau inequaltiy. \\

Now let us assume that the singular point is of type $D_{q,q_1}$. From the above discussion, we see that $K_X^2=1.$ Suppose that there are five exceptional curves over the non-cyclic quotient singular point. Then the orbifold Bogomolov-Miyaoka-Yau inequality tells us that
$$ 1 + \frac{(n_1-2)^2}{n_1}+ \frac{(n_2-2)^2}{n_2}+ \frac{(n_3-2)^2}{n_3} + m_1+m_2+b-2-4-a_2 = K^2_{X_0} $$
$$ \leq 3[\frac{1}{n_1}+\frac{1}{n_2}+\frac{1}{n_3}+\frac{1}{4((b-1)(m_1m_2-1)-m_2)(m_1m_2-1)}-1] $$
and we can check that the possible tuples of $(n_1, n_2, n_3, m_1, m_2, b)$ are $(2,2,2,2,2,2)$, $(2,2,2,2,2,3)$. We also need $|det(R)K_{X_0}^2|=|n_1n_2n_3det(R_x)K^2_{X_0}|$ to be a square number. Recall that $det(R_x) = (-4)((b-1)(m_1m_2-1)-m_2)$ and $a_2=1-\frac{(b-1)m_1-1}{(b-1)(m_1m_2-1)-m_2}$ and we can remove these cases from this restriction. \\

Now suppose that there are four exceptional curves over the non-cyclic quotient singular point. From the above discussion, we see that $K_X^2=1.$  Then the orbifold Bogomolov-Miyaoka-Yau inequality tells us that
$$ 1 - D^2_x = 1 - (a_0E_0 + a_1E_1 + a_2E_2 + a_3E_3)^2= 1+ba_0^2+2a_1^2+2a_2^2+ma_3^2-2a_0(a_1+a_2+a_3) $$
$$ = 1+ (b-1)a^2_0 + (ma_3-2a_0)a_3 \leq K^2_{X_0} \leq 3[\frac{1}{n_1}+\frac{1}{n_2}+\frac{1}{n_3}+\frac{1}{4q((b-1)q-q_1)}-1] $$
One can check that $ma_3-2a_0 = \frac{(mb-m-1)(m-3)+1}{mb-m-1} \geq 0$ if $m \geq 3$ or $m=b=2.$ When $m \geq 3,$ then we have
$$ 1+\frac{b-1}{4} \leq 1-D^2_x \leq K^2_{X_0} \leq 3[\frac{1}{n_1}+\frac{1}{n_2}+\frac{1}{n_3}+\frac{1}{4q((b-1)q-q_1)}-1] $$
$$ \leq 3[\frac{1}{2}+\frac{1}{2}+\frac{1}{3}+\frac{1}{4q((b-1)q-q_1)}-1] \leq 1+\frac{3}{8}$$
and the above inequality holds only if $b \leq 2.$ Therefore we only need to check the cases $m=2$ or $m \geq 3, b=2$ cases.
When $m=b=2$ then the non-cyclic quotient singular point is Gorenstein. Then the orbifold Bogomolov-Miyaoka-Yau inequality becomes of the following form.
$$ 1 + \frac{(n_1-2)^2}{n_1} + \frac{(n_2-2)^2}{n_2} + m_1+m_2+\frac{m_1+m_2+2}{m_1m_2-1} -6 = K^2_{X_0} $$
$$ \leq 3[\frac{1}{n_1}+\frac{1}{n_2}+\frac{1}{m_1m_2-1}+\frac{1}{8}-1] $$
Using Sage, one can check that the only possible case is $2A_1 \oplus A_2 \oplus D_4.$ We can see that this case cannot occur from \cite{HKO}. \\

When $m \geq 3, b=2$, we have the following inequality.
$$ 1 + \frac{(m-2)^2}{(m-1)^2}+\frac{(m-2)^3}{(m-1)^2} \leq 3[\frac{1}{n_1}+\frac{1}{n_2}+\frac{1}{n_3}+\frac{1}{4mq}-1] \leq \frac{15}{8} $$
The inequality holds only if $m=3.$ Then the orbifold Bogomolov-Miyaoka-Yau inequality becomes of the following form.
$$ 1 + \frac{(n_1-2)^2}{n_1} + \frac{(n_2-2)^2}{n_2} + m_1+m_2+\frac{m_1+m_2+2}{m_1m_2-1} -6 + \frac{1}{2} = K^2_{X_0} $$
$$ \leq 3[\frac{1}{n_1}+\frac{1}{n_2}+\frac{1}{m_1m_2-1}+\frac{1}{24}-1] $$
And we can check that there is no possible case using Sage. \\

When $m=2$, we have the following inequality.
$$ 1 + \frac{2(b-2)^2}{(2b-3)} \leq 3[\frac{1}{n_1}+\frac{1}{n_2}+\frac{1}{n_3}+\frac{1}{4mq}-1] \leq \frac{11}{8} $$
And we can check that there is no $b$ satisfying the inequality. \\

Therefore we obtain the desired conclusion.
\end{proof}

We can summarize the above discussions as follows.

\begin{theorem}
Let $X_0$ be a combinatorially minimal Mori dream surface with Picard number 1 with at worst quotient singularities whose minimal resolution is the minimal surfaces of general type with $p_g=0.$ Then the number of singular points is less than or equal to 4. Moreover when the number of singular points is 4, then the possible type of singularities of $X_0$ is one of the $2A_1 \oplus 2A_3, 4A_2.$
\end{theorem}

We can also obtain a list of the possible type of singularities $X_0$ can have when the number of singular points is less than or equal to three. However we have to consider many more complicated cases and it seems that we need to find stronger restrictions to remove redundant cases. We left this problem for future research.

\bigskip

\section{Singularities of combinatorially minimal models with \\ Picard number 2 having at worst quotient singularities}\label{s7}

In this section we study possible baskets of the singular points on the Picard number 2 combinatorially minimal models of minimal surface of general type with $p_g=0$ when they have only quotient singularities. Again, we often used Sage (cf. \cite{Sage}) to obtain possible types of singularities. From now on, we will consider the following situation and fix notation during this section: \\
  
{\bf Situation:} Let $X$ be a smooth minimal surfaces of general type with $p_g=0$ and let $X \to \cdots \to X_0$ be a sequence of combinatorial contractions where $X_0$ is a combinatorially minimal model with Picard number 2. Suppose that $X_0$ has at worst quotient singularities.
First, we can compute the Hilbert symbol of hyperbolic lattice as follows.

\begin{lemma}
Let $H$ be the hyperbolic lattice. For any positive integer $m,$ we have $\epsilon_p(mH)=1$ for any prime number $p$.
\end{lemma}
\begin{proof}
From basis change we see that $mH \sim 2mX_1^2-2mX_2^2.$ Because $(2m,-2m)_p=1$ (cf. \cite[3, Proposition 2]{Serre}) we get the result.
\end{proof}

The following Lemma also restricts the possible types of singularities.

\begin{lemma} Let $X_0$ be a combinatorially minimal Mori dream surfaces with Picard number 2 having quotient singularities. Then $det(R + mH)=det(R)det(mH)$ and $|det(R)|$ is a square number.
\end{lemma}
\begin{proof}
$R + mH$ is a overlattice of the unimodular lattice $H^2(X,\mathbb{Z})/tors.$ Therefore $|det(R + mH)|$ is a square number and it is easy to see that $|det(mH)|$ is also a square number. Therefore $|det(R)|$ is a square number. 
\end{proof}

Now let us study possible types of singular points of the combinatorially minimal Mori dream surfaces with Picard number 2. Our strategy is almost the same as the Picard number 1 case discussed in the previous section except we use $mH$ instead of $\langle K_X \rangle.$ Let us write our strategy as follows. \\
(1) Fix the number of singular points and $K_X^2.$ \\
(2) From the orbifold Bogomolov-Miyaoka-Yau inequality we have finite number of list of possible orders of local fundamental groups of singular points of $X_0.$ \\
(3) We can check whether $|det(R)|$ is a square number or not. \\
(4) Compute Hilbert symbol to check whether $R + mH$ can be embedded into a unimodular lattice of signature $(1,\rho(X)-1).$ \\
(5) Use algebro-geometric argument to remove the remaining cases or try to find supporting examples. \\
Again, we could not complete the last step and do not know whether our list is too large or not. We left this analysis for future research. \\

First, let us bound the number of singularities.

\begin{lemma}
The number of singular points of $X_0$ is less than or equal to 8 and the number of singular points on $X_0$ is 8 if and only if $X_0$ has $8A_1$ singularities.
\end{lemma}
\begin{proof}
From the orbifold Bogomolov-Miyaoka-Yau inequality we obtain the following inequality
$$ 0 \leq e_{orb}(X_0) = 4 - \sum_{x \in Sing(X_0)}(1-\frac{1}{|G_x|}) $$
and we have the conclusion.
\end{proof}

\begin{remark}\cite{MLP02}
As mentioned before the Enriques surfaces with 8 nodes is a combinatorially minimal Mori dream surface with $8A_1$ singularities. Therefore $8A_1$ singularity is supported by examples.
\end{remark}

From now on let $X_0$ be a combinatorially minimal Mori dream surface with $\rho(X_0)=2$ whose minimal resolution is a minimal surface of general type with $ p_g=0$. Then we have more restrictions.

\begin{lemma}
The number of singular points of $X_0$ is less than or equal to 6.
\end{lemma}
\begin{proof}
Because the Picard number of $X$ is less than 10, there are at most 7 singular points. Suppose that there are 7 singular points on $X_0.$ Then from the orbifold Bogomolov-Miyaoka-Yau inequality we obtain the following inequality.
$$ \frac{1}{3} \leq \frac{1}{3}K_{X}^2 \leq \frac{1}{3}K_{X_0}^2 \leq e_{orb}(X_0) = 4 - \sum_{x \in Sing(X_0)}(1-\frac{1}{|G_x|}). $$
Then the possible list of 7-tuple of the order of the local fundamental groups of singular points is $$(2^7),(2^6,3).$$ 
If the 7-tuple is $(2^7)$ then the lattice $(2^7) \oplus mH$ cannot be embedded into $(-1,1^8)$ since $2^7$ is not a square number. If the 7-tuple is $(2^6,3)$ then the lattice $(2^6,3) \oplus mH$ cannot be embedded into $(-1,1^8)$ since $2^6 \times 3$ is not a square number or it contradicts to the orbifold Bogomolov-Miyaoka-Yau inequality since $\frac{1}{3} < K_{X_0}^2$ holds.
\end{proof}

\begin{lemma}
The number of non-cyclic singular points of $X_0$ is at most one. When the number of singular points of $X_0$ is greater than or equal to five, then $X_0$ has only cyclic quotient singularities.
\end{lemma}
\begin{proof}
Let $X$ be the minimal resolution of $X_0$ and let $x$ be a non-cyclic singular point of $X_0$. Then number of exceptional components over $x$ is at least four. Note that the difference of the Picard number of $X$ and the Picard number of $X_0$ is less than or equal to seven. Therefore we have the desired result.
\end{proof}

If $X_0$ has 6 singular points then the orbifold Bogomolov-Miyaoka-Yau inequality gives restrictions of the possible list of 6-tuple of the order of the local fundamental groups of singular points as follows.

\begin{lemma}
If there are 6 singular points on $X_0$ then the possible list of 6-tuple of the order of the local fundamental groups of singular points is 
$$(2^6),(2^4,3,[2,2]).$$
\end{lemma}
\begin{proof}
Let $X_0$ be a surface with $\rho(X_0)=2$ having six quotient singularities and let $X \to X_0$ be its minimal resolution. The number of exceptional curves of $X \to X_0$ is 6 or 7. Suppose that the number of exceptional curves of $X \to X_0$ is 6. Then $K_X^2=2$ and we have the following inequality
$$ 2 = K_X^2 \leq K_{X_0}^2 \leq 3[4-6+\sum_{i=1}^6 \frac{1}{n_i}]. $$
Therefore we have
$$ \frac{8}{3} \leq \sum_{i=1}^6 \frac{1}{n_i}. $$
We may assume that $n_1 \leq \cdots \leq n_6.$ Let us assume that one of $n_i$ is not equal to two. If $n_1=n_2=n_3=2, n_4 \geq 3$ then the above inequality fails. If $n_1=n_2=n_3=n_4=2, n_5 \geq 3$ then there are two disjoint curves with self-intersection numbers $-n_5$ and $-n_6$ and we have the following inequality $$ 2 + \frac{(n_5-2)^2}{n_5} + \frac{(n_6-2)^2}{n_6} \leq K_{X_0}^2 \leq 3[4-6+\sum_{i=1}^6 \frac{1}{n_i}]. $$
which fails when $n_5 \geq 3.$
If $n_1=n_2=n_3=n_4=n_5=2, n_6 \geq 3$ then there is a curve with self-intersection $-n_6$ and we have the following inequality $$ 2 + \frac{(n_6-2)^2}{n_6} \leq K_{X_0}^2 \leq 3[4-6+\sum_{i=1}^6 \frac{1}{n_i}]. $$
Then we have $n_1=\cdots=n_5=2, n_6 = 3$ but $2^5 \cdot 3$ is not a square number. Therefore the only possible case is $(2^6).$ \\

Suppose that the number of exceptional curves of $X \to X_0$ is 7. Then $K_X^2=1$ and we have the following inequality
$$ 1 = K_X^2 \leq K_{X_0}^2 \leq 3[4-6+\sum_{i=1}^6 \frac{1}{n_i}] $$
Therefore we have
$$ \frac{7}{3} \leq \sum_{i=1}^6 \frac{1}{n_i}. $$

We may assume that $n_1 \leq \cdots \leq n_6.$ Let us assume that one of $n_i$ is not equal to two. If $n_1=n_2=2, n_3 \geq 3$ then the only possible case is $(2,2,3,3,3,3).$ However since the Picard number of $X$ is 9, $K_{X_0}^2$ is strictly bigger than $\frac{7}{3}$ and there is no possible case.

If $n_1=n_2=n_3=2$ and $n_4 \geq 3$ then there are two disjoint curves with self- intersection numbers $-n_i, -n_j$ and two curves with self-intersection numbers $-m_1,-m_2$ meeting at a point which contracts to three points and we have the following inequality 
$$ 1 + \frac{(n_i-2)^2}{n_i} + \frac{(n_j-2)^2}{n_j}+m_1+m_2+\frac{m_1+m_2+2}{m_1m_2-1}-6 \leq K_{X_0}^2 \leq 3[-\frac{1}{2}+\frac{1}{n_i}+\frac{1}{n_j}+\frac{1}{m_1m_2-1}]. $$
and we can easily see that there is no possible case.

If $n_1=n_2=n_3=n_4=2$ and $n_5 \geq 3$ then there is a curve with self-intersection $-n_5$ and two curves with self-intersection numbers $-m_1,-m_2$ meeting at a point which contracts to a point and we have the following inequality 
$$ 1 + \frac{(n_5-2)^2}{n_5} + m_1+m_2+\frac{m_1+m_2+2}{m_1m_2-1}-6 \leq K_{X_0}^2 \leq 3[\frac{1}{n_5}+\frac{1}{m_1m_2-1}]. $$
And we can see that the possible cases are $n_5=3, m_1=m_2=2.$

If $n_1=n_2=n_3=n_4=n_5=2$ and $n_6 \geq 3$ then there are two curves with self-intersections $-m_1,-m_2$ which contract to a singular point and we have the following inequality 
$$ 1 + m_1+m_2+\frac{m_1+m_2+2}{m_1m_2-1}-6 \leq K_{X_0}^2 \leq 3[\frac{1}{2}+\frac{1}{m_1m_2-1}]. $$
And we can check that there is no possible case.
\end{proof}

Indeed, the $(2^6)$ and $(2^4,3,[2,2])$ cases are supported by examples.

\begin{remark}
There are combinatorially minimal Mori dream surfaces with Picard number 2 with 6 singular points whose order of the local fundamental groups of singular points is $(2^6)$ or $(2^4,3,[2,2]).$ They are image of product-quotient surfaces with $K^2=1, 2.$ See \cite{BCGP,BP} for more details.
\end{remark}

Therefore we get the following result.

\begin{proposition}
Let $X_0$ be a combinatorial minimal Mori dream surfaces with Picard number 2 whose minimal resolution is a minimal surfaces of general type with $p_g=0.$ If there are 6 singular points on the combinatorial minimal Mori dream surface then the possible list of 6-tuple of the order of the local fundamental groups of singular points is $(2^6),(2^4,3,[2,2]).$
\end{proposition}

Again, if $X_0$ has 5 singular points then the orbifold Bogomolov-Miyaoka-Yau inequality gives restrictions of the possible list of 5-tuple of the order of the local fundamental groups of singular points.

\begin{lemma}
If there are 5 singular points on $X_0$ then the possible list of 5-tuple of the order of the local fundamental groups of singular points is 
$$(2^2,4,[2,2],[2,2]),(2^4,[2,2,2]),(2^3,4,[2,3,2]),(2^2,3,3,[2,2,2]).$$
\end{lemma}
\begin{proof}
Let $X_0$ be a surface with $\rho(X_0)=2$ having five quotient singular points and let $X \to X_0$ be its minimal resolution. The number of exceptional curves of $X \to X_0$ is five, six or seven. Suppose that the number of exceptional curves of $X \to X_0$ is five. Then $K_X^2=3,$ $\rho(X)=7$ and we have the following inequality
$$ 3 = K_X^2 \leq K_{X_0}^2 \leq 3[4-5+\sum_{i=1}^5 \frac{1}{n_i}]. $$
Therefore we have 
$$ 2 \leq \sum_{i=1}^5 \frac{1}{n_i}. $$

We may assume that $n_1 \leq \cdots \leq n_5.$ If $n_3 \geq 3$ then the only possible case is $(2,2,3,3,3).$ However, in this case $K_{X_0}^2$ is strictly bigger than $3$ and hence we can remove this case using the orbifold Bogomolov-Miyaoka-Yau inequality. Therefore we see that $n_3=2.$ Then the possible pairs $(n_4, n_5)$ are $(2,n_5),$ $(3,n_5)$ where $n_5 \leq 6$ and $(4,4).$ If $n_4=2,$ then we have $3+\frac{(n_5-2)^2}{n_5} \leq 3[1+\frac{1}{n_5}]$ and hence $n_5=2,3.$ However, the determinant of $R$ is not a square number so we can remove these cases. If $n_4=3,$ then we have $\frac{10}{3}+\frac{(n_5-2)^2}{n_5} \leq 3[\frac{5}{6}+\frac{1}{n_5}]$ and we can see that there is no possible case. We can remove $(4,4)$ case using the orbifold Bogomolov-Miyaoka-Yau inequality. Therefore we see that the number of exceptional curves is 6 or 7. \\

Suppose that the number of exceptional curves of $X \to X_0$ is 6. Then $K_X^2=2$ and we have the following inequality
$$ 2 = K_X^2 \leq K_{X_0}^2 \leq 3[4-5+\sum_{i=1}^5 \frac{1}{n_i}]. $$
Therefore we have
$$ 2+\frac{(n_1-2)^2}{n_1}+\frac{(n_2-2)^2}{n_2}+\frac{(n_3-2)^2}{n_3}+\frac{(n_4-2)^2}{n_4}+ m_1+m_2+\frac{m_1+m_2+2}{m_1m_2-1}-6 $$
$$ \leq \sum_{i=1}^4 \frac{3}{n_i} + \frac{3}{m_1m_2-1} -3. $$
And the determinant $det(R)=n_1n_2n_3n_4(m_1m_2-1)$ need to be a square number. And we can check that there is no possible case. \\

Suppose that the number of exceptional curves of $X \to X_0$ is 7. Then $K_X^2=1$ and we have the following inequality
$$ 1 = K_X^2 \leq K_{X_0}^2 \leq 3[4-5+\sum_{i=1}^5 \frac{1}{n_i}] $$

Suppose that there are two pairs of curves meeting at a point contracted to a singular point. In this case, we have
$$ 1+\frac{(n_1-2)^2}{n_1}+\frac{(n_2-2)^2}{n_2}+\frac{(n_3-2)^2}{n_3}+ m_1+m_2+\frac{m_1+m_2+2}{m_1m_2-1}-6+ l_1+l_2+\frac{l_1+l_2+2}{l_1l_2-1}-6 $$
$$ \leq \sum_{i=1}^3 \frac{3}{n_i} + \frac{3}{m_1m_2-1}+ \frac{3}{l_1l_2-1} -3. $$

And the determinant $det(R)=n_1n_2n_3(m_1m_2-1)(l_1l_2-1)$ need to be a square number. We can check that $(2^2,4,[2,2],[2,2])$ is the only possible case. \\

Suppose that there are three curves contracted to a singular point. In this case, we have
$$ 1+\frac{(n_1-2)^2}{n_1}+\frac{(n_2-2)^2}{n_2}+\frac{(n_3-2)^2}{n_3}+\frac{(n_4-2)^2}{n_4}+ m_1+m_2+m_3+\frac{m_1m_2+m_2m_3}{m_1m_2m_3-m_1-m_3}-8 $$
$$ \leq \sum_{i=1}^4 \frac{3}{n_i} + \frac{3}{m_1m_2m_3-m_1-m_3} -3. $$
And the determinant $det(R)=n_1n_2n_3n_4(m_1m_2m_3-m_1-m_3)$ need to be a square number. We can check that $(2^4,[2,2,2]),(2^3,4,[2,3,2]),(2^2,3,3,[2,2,2])$ are the only possible cases. \\

From the above discussion, we see that the only possible cases are
$$ (2^2,4,[2,2],[2,2]),(2^4,[2,2,2]),(2^3,4,[2,3,2]),(2^2,3,3,[2,2,2]). $$
\end{proof}

Then we can remove some of these cases using lattice theory.

\begin{lemma} Let $Q$ be the quadratic form corresponding to a Hirzebruch-Jung continued fraction $[2,2]$. Then we take an orthogonal basis so the quadratic form is given as $Q=-2X^2_1-\frac{3}{2}X^2_2$ and we have $d_p(Q)=3$ and $\epsilon_2(Q)=-1$ and $\epsilon_p(Q)=1$ for all $p \neq 2.$ Then the lattices $2A_1 \oplus 2A_2 \oplus \frac{1}{4}(1,1) \oplus mH$ cannot be embedded into $I_{1,9}.$ 
\end{lemma}
\begin{proof}
We have $\epsilon_3(-2,-2)=1$ and hence $d_3(2A_1)=4$ and $\epsilon_3(2A_1)=1.$ And because $\epsilon_3(3,3)=-1$ we have $d(2A_2)=9$ and $\epsilon_3(2A_2)=-1.$ From this and $\epsilon_3(4,9)=1$ we have $d(2A_1 \oplus 2A_2)=36$ and $\epsilon_3(2A_1 \oplus 2A_2)=-1.$ From $\epsilon_3(36,-4)=1$ we have $\epsilon_3(2A_1 \oplus 2A_2 \oplus \frac{1}{4}(1,1)) = -1.$
Finally, we have $ \epsilon_3(2A_1 \oplus 2A_2 \oplus \frac{1}{4}(1,1) \oplus mH) = -1$ but since $\epsilon_3(I_{1,9}) = 1,$ we see that the lattice cannot be embedded into $I_{1,9}.$
\end{proof}

\begin{lemma} Let $Q$ be the quadratic form corresponding to a Hirzebruch-Jung continued fraction $[2,2,2]$. Then we take an orthogonal basis so the quadratic form is given as $Q=-2X^2_1-\frac{3}{2}X^2_2-\frac{4}{3}X^2_3$ and we have $d_p(Q)=-4$ and $\epsilon_2(Q)=-1$ and $\epsilon_p(Q)=1$ for all $p \neq 2.$ We cannot embed $2A_1 \oplus 2 \times \frac{1}{3}(1,1) \oplus A_2 \oplus mH$ into $I_{1,9}.$
\end{lemma}
\begin{proof}
We have $\epsilon_3(-2,-2)=1$ and hence $d_3(2A_1)=4$ and $\epsilon_3(2A_1)=1.$ And because $\epsilon_3(4,-4)=1$ we have $d(2A_1 \oplus A_3)=-16$ and $\epsilon_3(2A_1 \oplus A_3)=1.$ Similarly, we have $d_3(2 \times \frac{1}{3}(1,1))=9,$ $\epsilon_3(2 \times \frac{1}{3}(1,1))=-1$ from $\epsilon_3(-3,-3)=-1.$ Then we have $d_3(A_1 \oplus 2 \times \frac{1}{3}(1,1) \oplus A_2) = -144,$ $\epsilon_3(A_1 \oplus 2 \times \frac{1}{3}(1,1) \oplus A_2) = -1 $ and finally we get $\epsilon_3(A_1 \oplus 2 \times \frac{1}{3}(1,1) \oplus A_2 \oplus mH) = -1.$
Therefore we see that the lattice cannot be embedded into $I_{1,9}$ since $\epsilon_3(I_{1,9}) = 1.$
\end{proof}

From the above discussion, we can remove $(2^2,4,[2,2],[2,2]),(2^2,3,3,[2,2,2]).$ cases using lattice theory.

\begin{proposition}
Let $X_0$ be a combinatorial minimal Mori dream surfaces with Picard number 2 coming from $X$ which is a minimal surfaces of general type with $p_g=0.$ If there are 5 singular points on the combinatorial minimal Mori dream surface then the possible list of 5-tuple of the order of the local fundamental groups of singular points is 
$$ (2^3,4,[2,3,2]), (2^4,[2,2,2]). $$
\end{proposition}

We can summarize the above discussions as follows.

\begin{theorem}
Let $X_0$ be a combinatorially minimal Mori dream surface whose minimal resolution is a smooth minimal surfaces of general type with $p_g=0.$ If the Picard number of $X_0$ is 2 then the number of singular points is less than 7. Moreover when the number of singular points is big we have the followings. \\
(1) If the number of singular points of $X_0$ is six, then the possible type of singularities of $X_0$ is
$$ (2^6), (2^4,3,[2,2]). $$
(2) If the number of singular points of $X_0$ is five, then then the possible type of singularities of $X_0$ is
$$ (2^3,4,[2,3,2]), (2^4,[2,2,2]). $$
\end{theorem}

We do not know whether the cases $(2^3,4,[2,3,2]), (2^4,[2,2,2])$ in the above list are supported by examples or not. 

\begin{remark}
There are many combinatorially minimal MDS of general type with $\rho=2$ having only quotient singularities. See \cite{BCGP, BP} for more details.
\end{remark}

We can also obtain a list of the possible type of singularities $X_0$ can have when the number of singular points is less than or equal to four. However there are many more complicated cases and it seems that one need to find stronger restrictions to remove redundant cases. We left this problem for future research.

\bigskip

\section{Singularities of combinatorially minimal models in general}\label{s8}

Let $X_0$ be a combinatorially minimal model of a smooth minimal surface of general type with $p_g=0$. In this section we discuss what kind of singularities $X_0$ can have by analyzing our previous examples. 

\begin{definition}
Let $(X_0,x)$ be a normal surface singularity and $f : X \to X_0$ be a weak resolution. Then the geometric genus of $(X_0,x)$ is defined as follows.
$$ p_g(X_0, x) = \mathrm{dim}_{\mathbb{C}}R^1 f_* \mathcal{O}_X $$
\end{definition}

\begin{definition}
Let $X_0$ be a normal projective surface. A singular point $x \in X_0$ is an elliptic singularity if $p_g(X_0,x)=1.$
\end{definition}

\begin{definition}
Let $X_0$ be a normal projective surface. A singular point $x \in X_0$ is simple elliptic if the exceptional curve of the minimal resolution $X \to X_0$ is a smooth projective curve of genus 1.
\end{definition}

\begin{definition}
Let $X_0$ be a normal projective surface. A singular point $x \in X_0$ is du Bois if for a good resolution $f : X \to X_0$ with a reduced exceptional divisor $E,$ we have a canonical isomorphism $R^1f_* \mathcal{O}_X \cong H^1(E,\mathcal{O}_E).$ 
\end{definition}

\begin{remark}
(1) A rational singularity is a Du Bois singularity. \\
(2) A log canonical singularity is a Du Bois singularity. Therefore a simple elliptic singularity is a Du Bois singularity.
\end{remark}

By analyzing examples we discussed so far, we have the following result.

\begin{proposition}
Let $X$ be a surface discussed in Theorem \ref{CMM:example} and Theorem \ref{CMM:example2}. Then there exist at least one combinatorially minimal $X_0$ having at worst log canonical singularities. 
\end{proposition}
\begin{proof}
Let $X$ be a fake projective plane or a reducible fake quadric. Therefore we have $X_0=X$ and combinatorially minimal models are smooth for these cases. When $X$ is an Inoue surface or a Chen's surface, we can contract a smooth curve of genus 1 to obtain $X_0.$ Therefore we see that there are at least one $X_0$ having log canonical singularity for these cases. When $X$ is a primary Burniat surface or a Kulikov surface, we can contract three disjoint smooth curves of genus 1 to obtain $X_0.$ Therefore we see that there are at least one $X_0$ having at worst log canonical singularities for these cases. When $X$ is a product-quotient surface with $K^2=6, G=D_4 \times \mathbb{Z}/2$ or $K^2=4, G=\mathbb{Z}/4 \times \mathbb{Z}/2,$ we can contract disjoint $(-2)$-curves to obtain $X_0.$ Therefore we see that there are at least one $X_0$ having at worst  log canonical singularities for these cases. When $X$ is a Burniat surface with $2 \leq K^2 \leq 5,$ then we can contract disjoint union of smooth rational curves and simple elliptic curves to obtain $X_0.$ From the above discussion, we obtain the desired conclusion for the examples in Theorem \ref{CMM:example}. For the examples of Theorem \ref{CMM:example2}, we can contract union of smooth rational curves and simple elliptic curves to obtain combinatorially minimal model. Therefore we see that for any $X$ we discussed in Theorem \ref{CMM:example} and Theorem \ref{CMM:example2}, there is at least one $X_0$ having at worst log canonical singularities.
\end{proof}

Then we have the following natural questions.

\begin{question}
Let $X$ be a smooth minimal surface of general type with $p_g=0.$ Can we find a combinatorially minimal model $X_0$ having at worst log canonical or Du Bois singularities?
\end{question}

If we can characterize which kind of singularities $X_0$ can have, then the next question would be how to use it to classify possible singular points $X_0$ can have. In order to do this we need to develop a general machinary which works  for a larger class of singularities. 

\begin{question}
Can we generalize the discussions of previous sections to a larger class of singularities? For example, can we generalize the orbifold Bogomolov-Miaoka-Yau inequality to log canonical singularities or Du Bois singularities? If it is possible, can we classify which type of singularities $X_0$ can have?
\end{question}

\bigskip

\section{Questions and further directions}\label{s9}

We have discussed combinatorially minimal models of smooth minimal surfaces of general type with $p_g=0$. There are many natural and interesting questions arise. Let us discuss some of them. \\

First of all, it is important to determine whether all smooth minimal surfaces of general type with $p_g=0$ are Mori dream surfaces or not. 

\begin{question}
Let $X$ be a smooth minimal surface of general type with $p_g=0.$ Can we describe $\Eff(X)$, $\Nef(X)$, $\SAmp(X)$? Is $X$ a Mori dream surface?
\end{question}

On the other hand, even if there is a minimal surface of general type with $p_g=0$ which is not a Mori dream space, it seems to be interesting and meaningful to ask whether we can contract negative curves on it to reach a combinatorially minimal model. In this case we can see that it is the minimal resolution of a combinatorially minimal Mori dream surface.

\begin{question}
Let $X$ be a minimal surfaces of general type with $p_g=0$. Can we contract negative curves on $X$ to obtain a combinatorially minimal Mori dream surface? In other words, can we get $X$ as the minimal resolution of a combinatorially minimal Mori dream surface $X_0$?
\end{question}

We have provided several examples of minimal surfaces of general type with $p_g=0$ which are Mori dream surfaces in \cite{FL, Keum-Lee1}. Then it is a natural question whether combinatorially minimal models of them provide new descriptions of them.

\begin{question}
Can we provide explicit descriptions of combinatorially minimal models of minimal surfaces of general type with $p_g=0$ we discussed so far? Can we provide new descriptions of these surfaces from this perspective?
\end{question}

Via studying these questions and singularities of combinatorially minimal models of minimal surfaces of general type, we hope to answer the following questions. 

\begin{question}
Can we classify Mori dream surfaces which are combinatorially minimal? Moreover can we classify smooth minimal Mori dream surfaces with $p_g=0$ via studying their combinatorially minimal models? Does every smooth minimal surface of general type with $p_g=0$ can be obtained as the minimal resolution of a combinatorially minimal Mori dream surface?
\end{question}

We finish the discussion with the following question and hope studying it will be an important step toward the answer of Mumford's original question.

\begin{question}
Can a computer classify Mori dream surfaces of general type with $p_g=0$?
\end{question}

\bigskip


\end{document}